\documentclass[a4paper, 12pt, reqno, english]{amsart}
\usepackage[utf8]{inputenc}
\usepackage[T1]{fontenc}
\usepackage{babel}

\usepackage[bookmarks = true, colorlinks, citecolor = blue, urlcolor = blue]{hyperref}
\usepackage[capitalize, nameinlink]{cleveref}
\usepackage[titletoc, toc, title]{appendix}

\usepackage{amssymb}
\usepackage{tikz}

\numberwithin{equation}{section}
\numberwithin{figure}{section}

\theoremstyle{plain}
\newtheorem{theorem}{Theorem}[section]

\theoremstyle{plain}
\newtheorem*{theorem*}{Theorem}

\theoremstyle{plain}
\newtheorem{proposition}[theorem]{Proposition}

\theoremstyle{plain}
\newtheorem{lemma}[theorem]{Lemma}

\theoremstyle{plain}
\newtheorem{corollary}[theorem]{Corollary}

\theoremstyle{definition}
\newtheorem{definition}[theorem]{Definition}

\theoremstyle{definition}

\theoremstyle{definition}

\theoremstyle{remark}
\newtheorem{remark}[theorem]{Remark}

\addto\extrasenglish{
  
}

\newcommand{\Uqg}{U_q(\mathfrak{g})}
\newcommand{\Uql}{U_q(\mathfrak{l})}
\newcommand{\Uqlss}{U_q(\mathfrak{l}_{ss})}
\newcommand{\Uqk}{U_q(\mathfrak{k})}

\newcommand{\Cqg}{\mathbb{C}_q[G]}

\newcommand{\lieup}{\mathfrak{u}_+}
\newcommand{\lieum}{\mathfrak{u}_-}

\newcommand{\Sq}{S_q(\mathfrak{u}_+)}
\newcommand{\Lq}{\Lambda_q(\mathfrak{u}_-)}

\newcommand{\tschubert}{U^\prime(w_\mathfrak{l})}

\newcommand{\Rmat}{\mathcal{R}}
\newcommand{\Cas}{\mathcal{C}}

\newcommand{\canonical}{\mathcal{I}}

\newcommand{\id}{\mathrm{id}}

\newcommand{\embU}{\phi_U}

\newcommand{\ad}{\mathrm{ad}}
\newcommand{\adT}{\widetilde{\mathrm{ad}}}

\newcommand{\qdiff}{Q}

\newcommand{\dolb}{\eth}
\newcommand{\spinbun}{\Gamma(\Omega)}
\newcommand{\algspace}{\mathcal{M}}

\newcommand{\LagGr}{\mathbb{C}_q[Sp(2) / U(2)]}

\begin{document}

\title[Parthasarathy formula for the quantum Lagrangian Grassmannian]{The Parthasarathy formula and a spectral triple for the quantum Lagrangian Grassmannian of rank two}

\author{Marco Matassa}

\address{OsloMet – storbyuniversitetet}

\email{marco.matassa@oslomet.no}

\begin{abstract}
We show that the Dolbeault--Dirac operator on the quantum Lagrangian Grassmannian of rank two, an example of a quantum irreducible flag manifold, satisfies an appropriate version of the Parthasarathy formula. We use this result to complete the proof that the candidate spectral triple for this space, as defined by Krähmer, is a spectral triple.
\end{abstract}

\maketitle

\section*{Introduction}

Let $G / K$ be a symmetric space. The \emph{Parthasarathy formula} relates the square of a Dirac-type operator $D$, defined on a certain bundle over $G / K$, with the quadratic Casimir of the Lie algebra $\mathfrak{g}$. See \cite{parthasarathy} for the original derivation and \cite{agricola} for a readable account.

It is well known that compact Lie groups admit \emph{$q$-deformations}, that is non-commutative algebras $\Cqg$ which are deformations of the classical rings and enjoy many similar properties. Many homogeneous spaces $G / K$ can be also quantized within this setting. For instance the case of \emph{generalized flag manifolds} is treated in \cite{stokman-flag}.
When these flag manifolds are irreducible, they can also be described as symmetric spaces, and coincide with the class of Hermitian symmetric spaces.
This class of quantum homogeneous spaces is particularly well-behaved: for instance they admit a canonical $q$-analogue of the de Rham complex, as shown by Heckenberger and Kolb in \cite{HeKo04, HeKo06}, which has the same graded dimension as in the classical case. We stress that this is not the case for general quantum spaces.

Dirac operators on quantum irreducible flag manifolds have been defined in \cite{dirac-flag} and revisited in \cite{KrTS15}.
A natural question is whether these operators admit a quantum analogue of the Parthasarathy formula.
One important application of such a formula would be to obtain \emph{spectral triples} corresponding to these quantum spaces. Recall that a spectral triple is the main ingredient in the framework of non-commutative geometry developed by Connes \cite{connes}. The Dirac operators defined in \cite{dirac-flag} can be used to construct candidate spectral triples, but in general it is not known if they have \emph{compact resolvent}, which is one of the key conditions to be satisfied.
This is known to be the case for quantum projective spaces, as originally shown in \cite{dd-proj} and then revisited in \cite{mat-proj} within the setting of \cite{KrTS15}. The strategy in these papers is to show that $D^2$ satisfies a formula of Parthasarathy-type. Hence having such a formula would allow to prove the result in full generality.

In this paper we give the first example of a quantum symmetric space of \emph{rank two} for which an appropriate version of the Parthasarathy formula holds.
The homogeneous space we consider here is an example of a \emph{Lagrangian Grassmannian}.
A (complex) Lagrangian Grassmannian is the smooth manifold of Lagrangian subspaces of a complex symplectic vector space of dimension $2 n$. As a homogeneous space it can be identified with $Sp(n) / U(n)$, where $Sp(n)$ is the compact real form of the complex Lie group $Sp(2n, \mathbb{C})$. The manifold $Sp(n) / U(n)$ has complex dimension $n (n + 1)/2$ and rank $n$, see \cite{helgason}. Hence we are considering the case $n = 2$, with corresponding quantum coordinate ring $\LagGr$.

We will consider a Dolbeault--Dirac $D$ element which, up to minor modifications, is the one introduced in \cite{KrTS15}. This element naturally acts on a certain space $\spinbun$, which classically corresponds to the bundle of anti-holomorphic forms on $Sp(2) / U(2)$. The goal is to compare this element with $\Cas \otimes 1$, where $\Cas$ is an appropriately defined central element of $U_q(\mathfrak{sp}_4)$, which plays the role of the quadratic Casimir. Our main result is the following.

\begin{theorem*}[\cref{thm:D-parthasarathy}]
The element $D^2$ coincides with $\Cas \otimes 1$ as operators on the space $\spinbun$, up to terms in the quantized Levi factor $\Uql$.
\end{theorem*}

We then use this result to show that $D$ has compact resolvent.
This provides the first complete construction of a spectral triple for a quantum symmetric space of rank two.\footnote{I have been told by Réamonn Ó Buachalla that he and collaborators have obtained similar results for some low-dimensional spaces by different means, but the details are not available yet.}

\begin{theorem*}[\cref{thm:spectral-triple}]
The Dolbeault--Dirac operator $D$ has compact resolvent. Hence we get a spectral triple for the quantum Lagrangian Grassmannian $\LagGr$.
\end{theorem*}

We close this introduction by comparing these results with those of \cite{mat-par}. In the cited paper we state a negative result concerning the Parthasarathy formula for $\LagGr$, the quantum homogeneous space being investigated here.
This apparent contradiction is resolved by the observation that we are looking at two slightly \emph{different versions} of the Parthasarathy formula.
In the paper \cite{mat-par} we consider the Dolbeault--Dirac operator from a purely algebraic point of view, that is as an element of $\Uqg \otimes \mathrm{End}(\Lq)$.
Then we look for an identity of the form $D^2 \sim C \otimes 1$, where $C$ is a central element and $\sim$ means neglecting terms in the Levi factor, and show that such an identity can not be satisfied. While this result is technically correct, it does not tell the full story: namely we want to consider $D$ as acting on an appropriate space, denoted by $\spinbun$ here.
Taking this into account we get more relations, which algebraically translate into taking a certain quotient of $\Uqg \otimes \mathrm{End}(\Lq)$.
The main result of the current paper is that the relation $D^2 \sim C \otimes 1$ does hold in this quotient, which we will denote by $\algspace$, and consequently for $D^2$ acting on the vector space $\spinbun$.

The paper is organized as follows. In \cref{sec:preliminaries} we discuss some preliminary material.
In \cref{sec:dirac-ops} we recall the construction of Dolbeault--Dirac operators on irreducible flag manifolds, following the paper \cite{KrTS15} but with some modifications.
In \cref{sec:lagrangian} we provide various details regarding the quantum Lagrangian Grassmannian under consideration.
In \cref{sec:algebras} we derive the relations for the relevant quantum symmetric and exterior algebras.
In \cref{sec:square} we obtain an explicit expression for the square of the Dolbeault--Dirac operator.
In \cref{sec:casimir} we derive a central element playing the role of the quadratic Casimir from the R-matrix.
In \cref{sec:rewriting} we rewrite this element in a different manner, to facilitate comparison with the Dolbeault--Dirac operator.
In \cref{sec:more-exterior} we provide more useful formulae regarding the exterior algebra.
In \cref{sec:parthasarathy} we prove our main result, namely a Parthasarathy formula for the Dolbeault--Dirac operator considered here.
Finally, in \cref{sec:spectral-triple} we use this result to construct a spectral triple on the Lagrangian Grassmannian.

\vspace{3mm}

{\footnotesize
\emph{Acknowledgements}.
I would like to thank Réamonn Ó Buachalla and Robert Yuncken for many interesting discussions regarding the topics of this paper.
}

\section{Preliminaries}
\label{sec:preliminaries}

In this section we review some preliminary material, while also fixing some notation regarding Lie algebras, parabolic subalgebras and quantized enveloping algebras.

\subsection{Lie algebras and parabolic subalgebras}

Let $\mathfrak{g}$ be a finite-dimensional complex simple Lie algebra with a fixed Cartan subalgebra $\mathfrak{h}$.
We denote by $\Delta(\mathfrak{g})$ the root system, by $\Delta^{+}(\mathfrak{g})$ a choice of positive roots and by $\Pi = \{ \alpha_{1}, \cdots, \alpha_{r} \}$ the corresponding simple roots. We denote by $P$ the weight lattice.
Finally we denote by $(\cdot, \cdot)$ the symmetric bilinear form on $\mathfrak{h}^{*}$ induced by the Killing form, normalized as to have $(\alpha, \alpha) = 2$ for short roots.

Next we review parabolic subalgebras of $\mathfrak{g}$, following the presentation given in \cite[Section 2.2]{KrTS15}.
Let $S \subset \Pi$ be a subset of the simple roots and define
\[
\Delta(\mathfrak{l}) := \mathrm{span}(S) \cap \Delta(\mathfrak{g}), \quad
\Delta(\mathfrak{u}_{+}) := \Delta^{+}(\mathfrak{g}) \backslash \Delta^{+}(\mathfrak{l}).
\]
In terms of these root spaces we define the subspaces
\[
\mathfrak{l} := \mathfrak{h} \oplus \bigoplus_{\alpha \in \Delta(\mathfrak{l})} \mathfrak{g}_{\alpha}, \quad
\mathfrak{u}_{\pm} := \bigoplus_{\alpha \in \Delta(\mathfrak{u}_{+})} \mathfrak{g}_{\pm \alpha}, \quad
\mathfrak{p} := \mathfrak{l} \oplus \mathfrak{u}_{+}.
\]
It follows from the definitions that $\mathfrak{l}$ and $\mathfrak{u}_{\pm}$ are Lie subalgebras of $\mathfrak{g}$.
We call $\mathfrak{p}$ the \emph{standard parabolic subalgebra} associated to $S$ (we will omit the dependence on $S$ in the following). The subalgebra $\mathfrak{l}$ is reductive and is called the \emph{Levi factor} of $\mathfrak{p}$, while $\mathfrak{u}_{+}$ is a nilpotent ideal of $\mathfrak{p}$ called the \emph{nilradical}.
We will denote by $\mathfrak{l}_{ss}$ the semisimple part of the Levi factor.
Also we will refer to the roots of $\Delta(\mathfrak{u}_{+})$ as the \emph{radical roots}.

We will consider \emph{cominuscule} parabolics: these have the property that all radical roots contain a certain simple root $\alpha_t$ with multiplicity $1$.
For these parabolic subalgebras we have the commutation relations $[\mathfrak{u}_{+}, \mathfrak{u}_{-}] \subset \mathfrak{l}$.
This follows from the general commutation relations $[E_\alpha, E_\beta] = c_{\alpha, \beta} E_{\alpha + \beta}$, together with the fact that $\alpha_t$ appears with multiplicity $1$.

The adjoint action of $\mathfrak{p}$ on $\mathfrak{g}$ descends to an action on $\mathfrak{g} / \mathfrak{p}$. The decomposition $\mathfrak{g} = \mathfrak{u}_{-} \oplus \mathfrak{p}$ gives $\mathfrak{g} / \mathfrak{p} \cong \mathfrak{u}_{-}$ as $\mathfrak{l}$-modules.
With respect to the Killing form of $\mathfrak{g}$, both $\mathfrak{u}_{+}$ and $\mathfrak{u}_{-}$ are isotropic and we have $\mathfrak{l} = \mathfrak{u}^{\perp}$, where $\mathfrak{u} = \mathfrak{u}_{+} \oplus \mathfrak{u}_{-}$.
The pairing $\mathfrak{u}_{-} \otimes \mathfrak{u}_{+} \to \mathbb{C}$ coming from the Killing form is non-degenerate, so that $\mathfrak{u}_{-}$ and $\mathfrak{u}_{+}$ are dual as $\mathfrak{l}$-modules.

\subsection{Quantum algebras}

Let $\mathfrak{g}$ be a complex Lie algebra. Let $0 < q < 1$ and write $q_i := q^{(\alpha_i, \alpha_i) / 2}$. We follow the conventions of \cite[Chapter 4]{jantzen}.
The \emph{quantized enveloping algebra} $\Uqg$ is the algebra with generators $K_i$, $E_i$ and $F_i$ and with relations
\[
K_i E_j = q^{(\alpha_i, \alpha_j)} E_j K_i, \quad
K_i F_j = q^{-(\alpha_i, \alpha_j)} F_j K_i, \quad
[E_i, F_j] = \delta_{i j} \frac{K_i - K_i^{-1}}{q_i - q_i^{-1}},
\]
plus the quantum analogue of the Serre relations. We will also use the "simply-connected" version, denoted by the same symbol, where we allow the Cartan elements $K_\lambda$ with $\lambda \in P$ and similar relations. We have the coproduct $\Delta : \Uqg \to \Uqg \otimes \Uqg$ given by
\[
\Delta(K_\lambda) = K_\lambda \otimes K_\lambda, \quad
\Delta(E_i) = E_i \otimes 1 + K_i \otimes E_i, \quad
\Delta(F_i) = F_i \otimes K_i^{-1} + 1 \otimes F_i.
\]
The corresponding antipode $S: \Uqg \to \Uqg$ is given by
\[
S(K_\lambda) = K_\lambda^{-1}, \quad
S(E_i) = - K_i^{-1} E_i, \quad
S(F_i) = - F_i K_i.
\]
Finally we will use the $*$-structure given by
\[
K_\lambda^* = K_\lambda, \quad
E_i^* = F_i K_i, \quad
F_i^* = K_i^{-1} E_i.
\]
We will also use the standard definitions of $q$-numbers and $q$-factorials
\[
[n]_q = \frac{q^n - q^{-n}}{q - q^{-1}}, \quad
[n]_q! = [n]_q \cdots [1]_q.
\]

Next consider a Levi subalgebra $\mathfrak{l} \subset \mathfrak{g}$, corresponding to a subset of simple roots $S \subset \Pi$.
This can be quantized straightforwardly, as in \cite[Section 4]{stokman-flag}. We define
\[
\Uql := \left\langle K_\lambda, E_i, F_i : \lambda \in P, \ i \in S \right\rangle \subset \Uqg,
\]
where $\langle \cdot \rangle$ denotes the algebra generated by these elements. It is easy to see that this is a Hopf $*$-subalgebra of $\Uqg$.
The subalgebra $\Uqlss \subset \Uqg$ is defined in a similar way.

We will also consider the \emph{quantum coordinate rings} $\Cqg$, with $G$ the complex Lie group integrating $\mathfrak{g}$. These are defined as the restricted duals of the quantized enveloping algebras $\Uqg$. By duality they inherit maps making them into Hopf $*$-algebras.
We will be very brief about these algebras, as they will not play a main role in the following. Let us mention that we have a canonical $\Uqg$-bimodule structure, given by
\[
(X a) (Y) := a(Y X), \quad (a X) (Y) := a(X Y), \quad a \in \Cqg, \ X, Y \in \Uqg.
\]
Next we consider quantum rings of homogeneous spaces. Suppose we have a Hopf $*$-subalgebra $\Uqk \subset \Uqg$ (more generally a coideal), which we interpret as the quantization of the Lie algebra $\mathfrak{k}$ of a subgroup $K \subset G$. Then we will write
\[
\mathbb{C}_q [G / K] := \{ a \in \Cqg : X a = \varepsilon(X) a, \ \forall X \in \Uqk \}.
\]

Finally, to shorten certain formulae we will use the notation
\[
Q := q - q^{-1}.
\]

\subsection{Quantum root vectors and commutation relations}

We now recall the notion of quantum root vectors, which are the analogue of root vectors inside the quantized enveloping algebra $\Uqg$. These can be defined in terms of certain automorphisms $T_i$, which have been introduced by Lusztig. We will follow the conventions of \cite[Section 8.14]{jantzen}.

Let $w_0$ be the longest word of the Weyl group of $\mathfrak{g}$ and let $w_0 = s_{i_1} \cdots s_{i_n}$ be a reduced decomposition. Then it is well-known that the positive roots $\Delta^+(\mathfrak{g})$ can be obtained as
\[
\beta_j := s_{i_1} \cdots s_{i_{j - 1}} (\alpha_{i_j}), \quad j = 1, \cdots, n.
\]
With this notation in place, we define the positive and negative quantum root vectors as
\[
E_{\beta_j} := T_{i_1} \cdots T_{i_{j - 1}} (E_{i_j}), \quad
F_{\beta_j} := T_{i_1} \cdots T_{i_{j - 1}} (F_{i_j}), \quad j = 1, \cdots, n.
\]
It is important to notice this definition depends on the choice of the decomposition for $w_0$ (in the classical case the dependence is only through signs). A PBW basis of $\Uqg$ is obtained from these elements and the Cartan elements $K_\lambda$, similarly to the classical case.

We now recall a general result regarding the commutation relations in $\Uqg$, see for instance \cite[Section I.6.10]{BrGo02} and references therein.
With notation as above, for $j < k$ we have
\begin{equation}
\label{eq:comm-rel-uqg}
E_{\beta_j} E_{\beta_k} - q^{(\beta_j, \beta_k)} E_{\beta_k} E_{\beta_j} = \sum_{a_{j + 1}, \cdots, a_{k - 1}} c_{a_{j + 1}, \cdots, a_{k - 1}} E_{\beta_{j + 1}}^{a_{j + 1}} \cdots E_{\beta_{k - 1}}^{a_{k - 1}}.
\end{equation}
Observe that the weight of the elements on both sides has to match. Hence on the right-hand side we can only have elements of total weight $\beta_j + \beta_k$.

\section{Dolbeault--Dirac operators}
\label{sec:dirac-ops}

In this section we will introduce Dolbeault--Dirac operators on quantum irreducible flag manifolds. Our definition essentially follows \cite{KrTS15}, with minor modifications.

\subsection{Symmetric and exterior algebras}

Let $V$ be a $\Uqg$-module. Then we have a corresponding \emph{quantum symmetric algebra} $S_q(V)$ and \emph{quantum exterior algebra} $\Lambda_q(V)$, as defined in \cite{BrZw08}. They are quadratic algebras and $\Uqg$-module algebras. These symmetric and exterior algebras are related by \emph{quadratic duality}: we have $S_q(V)^! \cong \Lambda_q(V^*)$, where $V^*$ is the dual of $V$ and $A^!$ denotes the quadratic dual (or Koszul dual) of the quadratic algebra $A$, see \cite[Proposition 2.11]{BrZw08}. We refer to the cited paper for other properties.

In general the algebras $S_q(V)$ and $\Lambda_q(V)$ do not have the same graded dimensions as their classical counterparts. When the dimensions coincide we speak of \emph{flat deformations}.
It is known that we have flat deformations for $V = \lieup$, an abelian nilradical corresponding to a cominuscule parabolic subalgebra $\mathfrak{p}$. This was shown in \cite{Zwi09}, see also \cite[Proposition 3.2]{KrTS15}. We summarize the two main properties holding in this case in the following theorem.

\begin{theorem}
Let $\lieup$ be an abelian nilradical. Then:
\begin{enumerate}
\item $\Sq$ is a flat deformation of $S(\lieup)$,
\item $\Sq$ is a Koszul algebra.
\end{enumerate}
\end{theorem}

Let $\lieum$ be the $\Uql$-module dual to $\lieup$ (see below). Then it follows from general duality that also $\Lq$ is a flat deformation and a Koszul algebra, see \cite[Corollary 3.4]{KrTS15}.

\subsection{Canonical element}

Let $\{ x_i \}_i \subset \lieup$ and $\{ y_i \}_i \subset \lieum$ be dual bases, where we consider $\lieum \cong \lieup^*$ as a $\Uql$-module via $(X y)(x) := y(S^{-1}(X) x)$. Equivalently, the evaluation map $y(x) = \langle x, y \rangle$ gives a $\Uql$-invariant pairing $\langle \cdot, \cdot \rangle : \lieup \otimes \lieum \to \mathbb{C}$, in the sense that
\[
\langle X_{(1)} \triangleright x, X_{(2)} \triangleright y \rangle = \varepsilon(X) \langle x, y \rangle, \quad \forall X \in \Uql.
\]
With this notation, we can consider the \emph{canonical element}
\[
\canonical := \sum_i x_i \otimes y_i \in \lieup \otimes \lieum.
\]
The element $\canonical$ does not depend on the choice of bases. Moreover it is $\Uql$-invariant, in the sense that $(X_{(2)} \otimes X_{(1)}) \canonical = \varepsilon(X) \canonical$. This condition is equivalent to
\begin{equation}
\label{eq:condition-I}
(X \otimes 1) \canonical = (1 \otimes S(X)) \canonical, \quad \forall X \in \Uql.
\end{equation}

Before proceeding, let us make an important remark on the quadratic duality giving $\Sq^! \cong \Lq$.
Since $\Sq$ is a $\Uql$-module algebra, we also want $\Lq$ to be a $\Uql$-module algebra.
Then, as pointed out in \cite[Remark 3.3]{KrTS15}, we want define the orthogonal complement of the relations of $\Sq$ by the pairing
\[
\langle x \otimes x^\prime, y \otimes y^\prime \rangle = \langle x, y^\prime \rangle \langle x^\prime, y \rangle, \quad
x, x^\prime \in \lieup, \ y, y^\prime \in \lieum.
\]
On the other hand, in the book \cite[Chapter 1, Section 2]{quadratic-algebras} quadratic duality is defined with respect to the pairing $\langle x \otimes x^\prime, y \otimes y^\prime \rangle = \langle x, y \rangle \langle x^\prime, y^\prime \rangle$. Hence the quadratic dual we use here coincides with the \emph{opposite algebra} of the one defined in the cited book.

Having dispensed with these details, we can consider the canonical element $\canonical$ as an element of $\Sq \otimes \Lq^\mathrm{op}$, where $\mathrm{op}$ denotes the opposite algebra.

\begin{lemma}
\label{lem:canonical-square}
Let $\canonical \in \Sq \otimes \Lq^\mathrm{op}$. Then we have $\canonical^2 = 0$.
\end{lemma}

\begin{proof}
This is a standard fact related to quadratic duality, see for instance \cite[Chapter 2, Section 3]{quadratic-algebras} and also \cite[Proposition 5.5]{KrTS15}.
\end{proof}

\subsection{Embeddings}

The quantum symmetric algebra $S_q(V)$ is defined in \cite{BrZw08} as a quotient of the tensor algebra $T(V)$. In the case of $V = \lieup$, it turns out that we can also identify $\Sq$ with a certain subalgebra of $\Uqg$, which is convenient for our purposes.

Choose a reduced decomposition of $w_0$ such that we have the factorization $w_0 = w_{0, \mathfrak{l}} w_\mathfrak{l}$, where $w_{0, \mathfrak{l}}$ is the longest word of the Weyl group of $\mathfrak{l}$ and $w_\mathfrak{l}$ is reduced. Let us write
\[
w_{0, \mathfrak{l}} = s_{i_1} \cdots s_{i_M}, \quad
w_\mathfrak{l} = s_{i_{M + 1}} \cdots s_{i_{M + N}}.
\]
Then the set of radical roots $\Delta(\lieup)$ can be enumerated by
\[
\xi_j := w_{0, \mathfrak{l}} s_{i_{M + 1}} \cdots s_{i_{M + j - 1}} (\alpha_{i_{M + j}}), \quad j = 1, \cdots, N.
\]
Consider the corresponding quantum root vectors $\{ E_{\xi_i} \}_{i = 1}^N$. Denote by $\tschubert \subset \Uqg$ the algebra generated by these elements, called the twisted Schubert cell in \cite{Zwi09}.

\begin{theorem}[{\cite[Main Theorem 5.6]{Zwi09}}]
\label{thm:zwi-schubert}
We have the following.
\begin{enumerate}
\item $\tschubert$ is a quadratic algebra with relations of the form
\[
E_{\xi_i} E_{\xi_j} - q^{(\xi_i, \xi_j)} E_{\xi_j} E_{\xi_i} = \sum_{i < a \leq b < j} c_{i j}^{a b} E_{\xi_a} E_{\xi_b}, \quad i < j.
\]
\item $\tschubert$ is invariant under the adjoint action of $\Uql$.
\item $\tschubert$ is isomorphic to $\Sq$ as a graded $\Uql$-module algebra.
\end{enumerate}
\end{theorem}

We will denote by $\embU : \Sq \to \Uqg$ the map giving the isomorphism in this theorem. We will also fix the basis $\{ x_i \}_i$ of $\lieup$ in such a way that $\embU(x_i) = E_{\xi_i}$.
We consider $\Uqg$ as a $\Uql$-module via the \emph{adjoint action} $\ad(X) Y := X_{(1)} Y S(X_{(2)})$ (we will also denote it by $\triangleright$).

Next, let $\gamma : \Lq \to \mathrm{End}(\Lq)$ be the \emph{right regular representation} of $\Lambda_q(\lieum)$ on itself, that is $\gamma(y) y^\prime := y^\prime \wedge y$ for $y, y^\prime \in \Lq$.
We consider $\mathrm{End}(\Lq)$ as a $\Uql$-module via
\[
\adT(X) T := X_{(2)} T S^{-1}(X_{(1)}), \quad X \in \Uql, \ T \in \mathrm{End}(\Lq).
\]

\begin{lemma}
\label{lem:equiv-maps}
The maps $\phi_U : \Sq \to \Uqg$ and $\gamma : \Lq \to \mathrm{End}(\Lq)$ are equivariant.
\end{lemma}

\begin{proof}
The equivariance of $\phi_U$ follows from the above theorem. To show the equivariance of $\gamma$, we use its definition and the fact that $\Lq$ is a $\Uql$-module algebra. We compute
\[
X \gamma(y) y^\prime = X (y^\prime \wedge y) = (X_{(1)} y^\prime) \wedge (X_{(2)} y) = \gamma(X_{(2)} y) X_{(1)} y^\prime.
\]
From this it immediately follows that
\[
\adT(X) \gamma(y) = X_{(2)} \gamma(y) S^{-1}(X_{(1)}) = \gamma(X_{(3)} y) X_{(2)} S^{-1}(X_{(1)}) = \gamma(X y),
\]
which shows that the map $\gamma$ is equivariant with respect to $\Uql$.
\end{proof}

\subsection{Definition of $D$}

We will follow the approach of \cite{KrTS15}, with minor modifications. Let $\lieup$ be an abelian nilradical and let $\lieum$ be its dual. We identify the quantum symmetric algebra $S_q(\lieup)$ with the algebra generated by the $\{ E_{\xi_i} \}_i$. We identify the quadratic dual with the quantum exterior algebra $\Lambda_q(\lieum)$ and denote by $\{ y_i \}_i$ the basis dual to $\{ E_{\xi_i} \}_i$.

\begin{definition}
With the above notation, we define the \emph{Dolbeault element} $\dolb \in \Uqg \otimes \mathrm{End}(\Lambda_q(\lieum))$ by $\dolb := (\embU \otimes \gamma) (\canonical)$. More explicitely, we have the expression
\[
\dolb = \sum_i \embU(x_i) \otimes \gamma(y_i) = \sum_i E_{\xi_i} \otimes \gamma(y_i).
\]
\end{definition}

It follows immediately from \cref{lem:canonical-square} that we have $\dolb^2 = 0$.

Next we choose a $\Uql$-invariant Hermitian inner product on $\Lq$, denoted by $(\cdot, \cdot)$. This means that for any elements $y, y^\prime \in \Lq$ we should have
\[
(X y, y^\prime) = (y, X^* y^\prime), \quad \forall X \in \Uql,
\]
where on the right-hand side $*$ denotes the $*$-structure of $\Uql$. From this inner product we get an adjoint operation on $\mathrm{End}(\Lq)$, which we also denote by $*$.

\begin{definition}
\label{def:dolb-dirac}
We define the \emph{Dolbeault--Dirac element} by $D := \dolb + \dolb^* \in \Uqg \otimes \mathrm{End}(\Lq)$.
\end{definition}

Observe that this definition depends on the choice of the Hermitian inner product on $\Lq$. It will turn out, in the example treated in this paper, that there is an essentially unique choice which gives a simple expression for $D^2$. A similar situation arises in \cite{mat-proj}.

\begin{remark}
We compare our definition with the one given in \cite{KrTS15}. In the cited paper we have $D_{\mathrm{KTS}} = \eth_{\mathrm{KTS}} + \eth_{\mathrm{KTS}}^*$, where $\eth_{\mathrm{KTS}} = \sum_i S^{-1}(E_{\xi_i}) \otimes \gamma_-(y_i)$. Here $\gamma_-(\cdot)$ are the analogue of contraction operators in the exterior algebra, and are defined by duality from the left regular representation $\gamma_+$. The element $\eth_{\mathrm{KTS}}$ can be identified with the Koszul differential of $\Sq^{\mathrm{op}} \otimes \Lq$, and classically reduces to the adjoint of the Dolbeault operator $\bar{\partial}$.

In our case, the Dolbeault element is identified with the Koszul differential of $\Sq \otimes \Lq^{\mathrm{op}}$, which is why we use the right regular representation $\gamma$. Moreover in the classical case it gives the Dolbeault operator, as opposed to its adjoint.
\end{remark}


\subsection{Spinor bundle}

The Dolbeault--Dirac element $D$ naturally acts on $\Cqg \otimes \Lq$.
However we will consider its action on a proper subspace, defined below.

\begin{definition}
We define $\spinbun := (\Cqg \otimes \Lq)^{\Uql}$. More explicitly we have
\[
\spinbun = \{ \xi \in \Cqg \otimes \Lq : X \xi = \varepsilon(X) \xi, \ \forall X \in \Uql \},
\]
where the action of $\Uql$ on the tensor product is given by $X (a \otimes y) = X_{(2)} a \otimes X_{(1)} y$.
\end{definition}

Here we are using the same conventions used in \cite[Section 6]{dirac-flag}.
Observe that, for an element $\xi \in \Cqg \otimes \Lq$, the condition $\xi \in \spinbun$ is equivalent to
\begin{equation}
\label{eq:spin-bundle}
(X \otimes 1) \xi = (1 \otimes S(X)) \xi, \quad \forall X \in \Uql.
\end{equation}
In the classical case, the space $\spinbun$ can be identified with the sections of the bundle of anti-holomorphic forms on $G / P$, for an appropriate complex structure.

We will now show that the Dolbeault--Dirac element naturally acts on $\spinbun$.

\begin{proposition}
The action of the Dolbeault element $\dolb$ maps $\spinbun$ into itself. Hence the Dolbeault--Dirac element $D$ acts on $\spinbun$.
\end{proposition}

\begin{proof}
To prove that $\dolb$ maps $\spinbun$ into itself it suffices to show that $(X \otimes 1) \dolb \xi = (1 \otimes S(X)) \dolb \xi$ for all $X \in \Uql$ and $\xi \in \spinbun$.
We start by considering the identity
\[
(X \otimes 1) \dolb = (X_{(1)} \otimes 1) \dolb (S(X_{(2)}) X_{(3)} \otimes 1) = (\ad(X_{(1)}) \otimes \id) (\dolb) (X_{(2)} \otimes 1).
\]
Recall that we have defined $\dolb = (\embU \otimes \gamma) (\canonical)$ and that $(X \otimes 1) \canonical = (1 \otimes S(X)) \canonical$ by \eqref{eq:condition-I}. Since the maps $\embU$ and $\gamma$ are equivariant as in \cref{lem:equiv-maps}, we have
\[
(\ad(X) \otimes \id) (\dolb) = (\id \otimes \adT(S(X))) (\dolb).
\]
Using this identity we can rewrite
\[
(X \otimes 1) \dolb = (\id \otimes \adT(S(X_{(1)}))) (\dolb) (X_{(2)} \otimes 1) = (1 \otimes S(X_{(1)})) \dolb (X_{(3)} \otimes X_{(2)}).
\]
Finally acting on $\xi \in \spinbun$ and using the condition \eqref{eq:spin-bundle} we obtain
\[
(X \otimes 1) \dolb \xi = (1 \otimes S(X_{(1)})) \dolb (X_{(3)} S^{-1}(X_{(2)}) \otimes 1) \xi = (1 \otimes S(X)) \dolb \xi.
\]
This shows that $\dolb$ maps $\spinbun$ into itself. The same is true for $\dolb^*$, since this is the adjoint of $\dolb$ with respect to a $\Uql$-invariant inner product, and hence for $D$.
\end{proof}

\section{Lagrangian Grassmannian}
\label{sec:lagrangian}

In this section we will provide some details regarding the quantum Lagrangian Grassmannian $\LagGr$.
We describe in particular the $\Uql$-module $\lieup$, corresponding to its tangent space, and well as its corresponding quantum root vectors.

\subsection{Lie algebra $C_2$}

Consider the complex simple Lie algebra $C_2 = \mathfrak{sp}_4$. The simple roots are $\{ \alpha_1, \alpha_2 \}$ and we choose the standard convention in which $\alpha_1$ is short and $\alpha_2$ is long. Corresponding to this choice, and fixing $(\alpha_1, \alpha_1) = 2$, we have
\[
(\alpha_1, \alpha_1) = 2, \quad
(\alpha_1, \alpha_2) = -2, \quad
(\alpha_2, \alpha_2) = 4.
\]
We will consider the fundamental representation $V(\omega_1)$. This will be used later on to construct a central element of $U_q(\mathfrak{sp}_4)$ playing the role of the quadratic Casimir. The weights are
\[
\lambda_1 := \omega_1 = \alpha_1 + \frac{1}{2} \alpha_2, \quad
\lambda_2 := - \omega_1 + \omega_2 = \frac{1}{2} \alpha_2, \quad
\lambda_3 := - \lambda_2, \quad
\lambda_4 := - \lambda_1.
\]
We fix a weight basis $\{ v_i \}_{i = 1}^4$, where the vector $v_i$ has weight $\lambda_i$.

\begin{lemma}
\label{lem:fundamental-c2}
The action of $U_q(\mathfrak{sp}_4)$ on $V(\omega_1)$ is given by
\[
\begin{gathered}
K_1 v_1 = q v_1, \quad K_1 v_2 = q^{-1} v_2, \quad K_1 v_3 = q v_3, \quad K_1 v_4 = q^{-1} v_4, \\
K_2 v_1 = v_1, \quad K_2 v_2 = q^2 v_2, \quad K_2 v_3 = q^{-2} v_3, \quad K_2 v_4 = v_4, \\
E_1 v_2 = q^{1/2} v_1, \quad E_1 v_4 = q^{1/2} v_3, \quad E_2 v_3 = q v_2, \\
F_1 v_1 = q^{-1/2} v_2, \quad F_1 v_3 = q^{-1/2} v_4, \quad F_2 v_2 = q^{-1} v_3,
\end{gathered}
\]
and with all the other elements equal to zero.
\end{lemma}

\begin{proof}
Follows from simple computations.
\end{proof}

\subsection{Parabolic subalgebra}

We consider the cominuscule parabolic subalgebra obtained from $S = \{ \alpha_1 \}$, that is by \emph{removing} the long root $\alpha_2$. The nilradical is then given by
\[
\lieup = \mathrm{span} \{ e_{\alpha_2}, \ e_{\alpha_1 + \alpha_2}, \ e_{2 \alpha_1 + \alpha_2} \}.
\]
The Levi factor $\mathfrak{l} = \mathrm{span} \{ h_1, h_2, e_1, f_1 \}$ can be identified with $\mathfrak{gl}_2$. Its semi-simple part $\mathfrak{l}_{ss} = \mathrm{span} \{ h_1, e_1, f_1 \}$ can be identified with $\mathfrak{sl}_2$. Also the nilradical $\lieup$ can be identified with the adjoint representation of $\mathfrak{sl}_2$ (we will rederive this result later in the quantum case).

Let $w_0$ be the longest word of the Weyl group of $\mathfrak{sp}_4$. Consider the reduced decomposition $w_0 = s_1 s_2 s_1 s_2$, which factorizes as $w_0 = w_{0, \mathfrak{l}} w_\mathfrak{l}$, where $w_{0, \mathfrak{l}} = s_1$ is the longest word corresponding to $\mathfrak{sl}_2$ and $w_\mathfrak{l} = s_2 s_1 s_2$. With this decomposition we obtain the positive roots
\[
\beta_1 = \alpha_1, \quad \beta_2 = 2 \alpha_1 + \alpha_2, \quad \beta_3 = \alpha_1 + \alpha_2, \quad \beta_4 = \alpha_2.
\]
The radical roots corresponding to $\lieup$ are then given by
\[
\xi_1 = \beta_2 = 2 \alpha_1 + \alpha_2, \quad
\xi_2 = \beta_3 = \alpha_1 + \alpha_2, \quad
\xi_3 = \beta_4 = \alpha_2.
\]

The corresponding homogeneous space $G / P$ is an example of a Lagrangian Grassmannian, as mentioned in the introduction. In general we have $G / P \cong G_0 / L_0$, where $G_0$ is the compact real form of $G$ and $L_0 = G \cap L$, with $L$ the Levi factor. Then $G_0 / L_0 = Sp(2) / U(2)$, where $Sp(n)$ is the compact real form of the complex Lie group $Sp(2n, \mathbb{C})$.

In the following we will frequently use the general notation
\[
\Uqg = U_q(\mathfrak{sp}_4), \quad \Uql = U_q(\mathfrak{gl}_2), \quad \Uqlss = U_q(\mathfrak{sl}_2),
\]
as it clarifies the role of the different algebras in the various steps.

\subsection{Quantum root vectors}

We will now obtain explicit expressions for the quantum root vectors, which are defined using the reduced decomposition $w_0 = s_1 s_2 s_1 s_2$.

\begin{lemma}
\label{lem:root-E}
1) The quantum root vectors $\{ E_{\beta_i} \}_{i = 1}^4$ are given by
\[
\begin{gathered}
E_{\beta_1} = E_1, \quad
E_{\beta_2} = \frac{1}{[2]_q} E_1^2 E_2 - q^{-1} E_1 E_2 E_1 + \frac{q^{-2}}{[2]_q} E_2 E_1^2, \\
E_{\beta_3} = E_1 E_2 - q^{-2} E_2 E_1, \quad
E_{\beta_4} = E_2.
\end{gathered}
\]
2) The quantum root vectors $\{ F_{\beta_i} \}_{i = 1}^4$ are given by
\[
\begin{gathered}
F_{\beta_1} = F_1, \quad
F_{\beta_2} = \frac{1}{[2]_q} F_2 F_1^2 - q F_1 F_2 F_1 + \frac{q^2}{[2]_q} F_1^2 F_2, \\
F_{\beta_3} = F_2 F_1 - q^2 F_1 F_2, \quad
F_{\beta_4} = F_2.
\end{gathered}
\]
\end{lemma}

\begin{proof}
1) This follows from the computations given in \cite[Subsection 8.17]{jantzen}, keeping in mind our choice for the reduced decomposition of $w_0$.

2) This follows by applying the anti-automorphism $\Omega : \Uqg \to \Uqg$, which is defined by
\[
\Omega(K_\lambda) = K_\lambda^{-1}, \quad
\Omega(E_i) = F_i, \quad
\Omega(F_i) = E_i, \quad
\Omega(q) = q^{-1}.
\]
Since $\Omega$ commutes with the Lusztig automorphisms $T_i$, we obtain the result.
\end{proof}

\section{The symmetric and exterior algebras}
\label{sec:algebras}

In this section we will derive various properties of the algebras $\Sq$ and $\Lq$, in the case when $\lieup$ is the nilradical corresponding to the Lagrangian Grassmannian.

\subsection{Relations for $\Sq$}

Consider the the algebra generated by the quantum root vectors $\{ E_{\xi_i} \}_{i = 1}^3$ in $\Uqg$. Then according to \cref{thm:zwi-schubert} this is a quadratic algebra isomorphic to $\Sq$. In the following we will make this identification and simply denote the former algebra by $\Sq$. We start by determining these relations. Recall that $Q = q - q^{-1}$.

\begin{proposition}
The algebra $S_q(\mathfrak{u}_+)$ has generators $\{ E_{\xi_i} \}_{i = 1}^3$ and relations
\[
E_{\xi_1} E_{\xi_2} = q^2 E_{\xi_2} E_{\xi_1}, \quad
E_{\xi_2} E_{\xi_3} = q^2 E_{\xi_3} E_{\xi_2}, \quad
E_{\xi_1} E_{\xi_3} = E_{\xi_3} E_{\xi_1} + Q \frac{q}{[2]_q} E_{\xi_2}^2.
\]
\end{proposition}

\begin{proof}
Using the general commutation relations given in \eqref{eq:comm-rel-uqg}, we immediately obtain the first two identities. For the third case, we must have $E_{\xi_1} E_{\xi_3} - E_{\xi_3} E_{\xi_1} = c E_{\xi_2}^2$ by weight reasons.
The constant $c$ must be determined by explicit computation. We compute
\[
\begin{split}
[2]_q E_{\beta_4} E_{\beta_2} & = E_{\beta_4} E_{\beta_1} E_{\beta_3} - E_{\beta_4} E_{\beta_3} E_{\beta_1} \\
& = q^2 E_{\beta_1} E_{\beta_4} E_{\beta_3} - q^2 E_{\beta_3}^2 - q^{-2} E_{\beta_3} E_{\beta_4} E_{\beta_1} \\
& = E_{\beta_1} E_{\beta_3} E_{\beta_4} - q^2 E_{\beta_3}^2 - E_{\beta_3} E_{\beta_1} E_{\beta_4} + E_{\beta_3}^2 \\
& = [2]_q E_{\beta_2} E_{\beta_4} - (q - q^{-1}) q E_{\beta_3}^2.
\end{split}
\]
From this we conclude that
\[
[E_{\xi_1}, E_{\xi_3}] = [E_{\beta_2}, E_{\beta_4}] = (q - q^{-1}) \frac{q}{[2]_q} E_{\beta_3}^2 = Q \frac{q}{[2]_q} E_{\xi_2}^2. \qedhere
\]
\end{proof}

\subsection{Action of $\Uql$ on $\lieup$}

Recall that $\Uqg$ acts on itself by the (left) adjoint action, which here we denote by $X \triangleright Y = X_{(1)} Y S(X_{(2)})$. We restrict this to an action of the quantized Levi factor $\Uql \subset \Uqg$ and compute this action on the nilradical $\lieup \subset \Uqg$.

\begin{lemma}
\label{lem:levi-up}
The action of $\Uql$ on $\lieup$ is given by
\[
\begin{gathered}
K_1 \triangleright E_{\xi_1} = q^2 E_{\xi_1}, \quad
K_1 \triangleright E_{\xi_2} = E_{\xi_2}, \quad
K_1 \triangleright E_{\xi_3} = q^{-2} E_{\xi_3}, \\
K_2 \triangleright E_{\xi_1} = E_{\xi_1}, \quad
K_2 \triangleright E_{\xi_2} = q^2 E_{\xi_2}, \quad
K_2 \triangleright E_{\xi_3} = q^4 E_{\xi_3}, \\
E_1 \triangleright E_{\xi_1} = 0, \quad
E_1 \triangleright E_{\xi_2} = [2]_q E_{\xi_1}, \quad
E_1 \triangleright E_{\xi_3} = E_{\xi_2}, \\
F_1 \triangleright E_{\xi_1} = E_{\xi_2}, \quad
F_1 \triangleright E_{\xi_2} = [2]_q E_{\xi_3}, \quad
F_1 \triangleright E_{\xi_3} = 0.
\end{gathered}
\]
\end{lemma}

\begin{proof}
The action of the Cartan elements is obtained immediately. Next observe that in our conventions we have $E_i \triangleright E_{\beta_j} = E_i E_{\beta_j} - q^{(\alpha_i, \beta_j)} E_{\beta_j} E_i$.
Then it follows from the commutation relations \eqref{eq:comm-rel-uqg} and the expressions given in \cref{lem:root-E} that we have
\[
E_{\beta_1} \triangleright E_{\beta_2} = 0, \quad
E_{\beta_1} \triangleright E_{\beta_3} = [2]_q E_{\beta_2}, \quad
E_{\beta_1} \triangleright E_{\beta_4} = E_{\beta_3}.
\]
Next using $F_i \triangleright E_{\beta_j} = [F_i, E_{\beta_j}] K_i$ it is clear that $F_1 \triangleright E_{\beta_4} = 0$. Then we have
\[
F_1 \triangleright E_{\beta_3} = F_1 E_1 \triangleright E_{\beta_4} = - \frac{K_1 - K_1^{-1}}{q - q^{-1}} \triangleright E_{\beta_4} = [2]_q E_{\beta_4},
\]
where we have used that $F_1 \triangleright E_{\beta_4} = 0$. Finally we have
\[
F_1 \triangleright E_{\beta_2} = \frac{1}{[2]_q} F_1 E_1 \triangleright E_{\beta_3} = \frac{1}{[2]_q} E_1 F_1 \triangleright E_{\beta_3} = E_1 \triangleright E_{\beta_4} = E_{\beta_3},
\]
where we have used $(K_1 - K_1^{-1}) \triangleright E_{\beta_3} = 0$, which follows from $(\alpha_1, \beta_3) = 0$.
\end{proof}

In the case under consideration we have $\mathfrak{l} = \mathfrak{gl}_2$ with semi-simple part $\mathfrak{l}_{ss} = \mathfrak{sl}_2$. We can identify $\lieup$ with the $U_q(\mathfrak{sl}_2)$-module of highest weight $2 \omega_1$, that is the adjoint representation.

\subsection{Relations for $\Lq$}

Recall that we identify $\Sq$ with the algebra generated by the quantum root vectors $\{ E_{\xi_i} \}_{i = 1}^3$ in $\Uqg$. Then we will identify $\Lq$ with the quadratic dual of the latter algebra.
We denote by $\{ y_i \}_{i = 1}^3$ the dual basis to $\{ E_{\xi_i} \}_{i = 1}^3$ with respect to the dual pairing $\langle \cdot, \cdot \rangle : \lieup \otimes \lieum \to \mathbb{C}$, which is extended to $\lieup^{\otimes 2} \otimes \lieum^{\otimes 2}$ by the formula
\[
\langle x \otimes x^\prime, y \otimes y^\prime \rangle = \langle x, y^\prime \rangle \langle x^\prime, y \rangle, \quad
x, x^\prime \in \lieup, \ y, y^\prime \in \lieum.
\]

\begin{proposition}
The algebra $\Lq$ has generators $\{ y_i \}_{i = 1}^3$ and relations
\[
\begin{gathered}
y_1 \wedge y_1 = 0, \quad
y_2 \wedge y_2 = Q \frac{q}{[2]_q} y_1 \wedge y_3, \quad
y_3 \wedge y_3 = 0, \\
y_1 \wedge y_2 = -q^2 y_2 \wedge y_1, \quad
y_1 \wedge y_3 = -y_3 \wedge y_1, \quad
y_2 \wedge y_3 = -q^2 y_3 \wedge y_2.
\end{gathered}
\]
\end{proposition}

\begin{proof}
The algebra $\Sq$ has the relations $R = \mathrm{span} \{ X_1, X_2, X_3 \} \subset \lieup \otimes \lieup$, where
\[
\begin{gathered}
X_1 = x_1 \otimes x_2 - q^2 x_2 \otimes x_1, \quad
X_2 = x_2 \otimes x_3 - q^2 x_3 \otimes x_2, \\
X_3 = x_1 \otimes x_3 - x_3 \otimes x_1 - Q \frac{q}{[2]_q} x_2 \otimes x_2.
\end{gathered}
\]
We need to determine the subspace $R^\perp \subset \lieum \otimes \lieum$ such that $\langle R, R^\perp \rangle = 0$. Since $\langle x_i, y_j \rangle = \delta_{i j}$, it is clear that $y_1 \otimes y_1 \in R^\perp$ and $y_3 \otimes y_3 \in R^\perp$. On the other hand we claim that
\[
Y = y_2 \otimes y_2 - Q \frac{q}{[2]_q} y_1 \otimes y_3 \in R^\perp.
\]
It is clear that the only possibly non-zero pairing is $\langle X_3, Y \rangle$. We have
\[
\begin{split}
\langle X_3, Y \rangle & = \langle - Q \frac{q}{[2]_q} x_2 \otimes x_2, y_2 \otimes y_2 \rangle + \langle - x_3 \otimes x_1, - Q \frac{q}{[2]_q} y_1 \otimes y_3 \rangle \\
& = - Q \frac{q}{[2]_q} + Q \frac{q}{[2]_q} = 0.
\end{split}
\]
By similar computations one shows that the elements $y_1 \otimes y_2 + q^2 y_2 \otimes y_1$, $y_1 \otimes y_3 + y_3 \otimes y_1$ and $y_2 \otimes y_3 + q^2 y_3 \otimes y_2$ all belong to $R^\perp$. The above elements span $R^\perp$ by dimensional reasons.
\end{proof}

\subsection{Action of $\Uql$ on $\lieum$}

The vector space $\lieum$ carries a natural action of $\Uql$, being defined as the dual of $\lieup$ in terms of the $\Uql$-invariant pairing $\langle \cdot, \cdot \rangle : \lieup \otimes \lieum \to \mathbb{C}$.

\begin{lemma}
\label{lem:levi-um}
The action of $\Uql$ on $\lieum$ is given by
\[
\begin{gathered}
K_1 y_1 = q^{-2} y_1, \quad
K_1 y_2 = y_2, \quad
K_1 y_3 = q^2 y_3, \\
K_2 y_1 = y_1, \quad
K_2 y_2 = q^{-2} y_2, \quad
K_2 y_3 = q^{-4} y_3, \\
E_1 y_1 = - [2]_q y_2, \quad
E_1 y_2 = - q^2 y_3, \quad
E_1 y_3 = 0, \\
F_1 y_1 = 0, \quad
F_1 y_2 = - y_1, \quad
F_1 y_3 = - [2]_q q^{-2} y_2.
\end{gathered}
\]
\end{lemma}

\begin{proof}
Since the dual pairing $\langle \cdot, \cdot \rangle : \lieup \otimes \lieum \to \mathbb{C}$ is $\Uql$-invariant, we get $\langle K_k x_i, K_k y_j \rangle = \langle x_i, y_j \rangle$, which gives the expression for the Cartan elements. For $E_1$ and $F_1$ we have
\[
\langle E_1 x_i, y_j \rangle + \langle K_1 x_i, E_1 y_j \rangle = 0, \quad
\langle F_1 x_i, K_1^{-1} y_j \rangle + \langle x_i, F_1 y_j \rangle = 0.
\]
Using the formulae from \cref{lem:levi-up} we easily get the result.
\end{proof}

\section{Square of the Dolbeault--Dirac element}
\label{sec:square}

In this section we will obtain an explicit expression for $D^2$, the square of the Dolbeault--Dirac element. We will derive this expression up to terms in the quantized Levi factor $\Uql$.

\subsection{Some commutation relations}

To simplify $D^2$, we will need commutation relations between the quantum root vectors $E_{\xi_i}$ and $E_{\xi_j}^*$. For convenience, we will derive these modulo elements in the quantized Levi factor $\Uql$, using the following notation.

\begin{definition}
\label{def:equiv-levi}
We define an equivalence relation $\sim$ on $\Uqg$ as follows: we will write $X \sim Y$ if $X - Y = Z$ for some $Z \in \Uql$.
\end{definition}

Observe that if $X \sim Y$ then $X^* \sim Y^*$, and also $Z \triangleright X \sim Z \triangleright Y$ for all $Z \in \Uql$. These properties follow from the fact that $\Uql$ is a Hopf $*$-subalgebra of $\Uqg$.

\begin{lemma}
\label{lem:rel-xi-xis}
We have the relations
\[
\begin{gathered}
E_{\xi_1} E_{\xi_1}^* \sim q^{-4} E_{\xi_1}^* E_{\xi_1} - Q q^{-2} E_{\xi_2}^* E_{\xi_2} + Q^2 [2]_q q^{-3} E_{\xi_3}^* E_{\xi_3}, \\
E_{\xi_2} E_{\xi_2}^* \sim q^{-2} E_{\xi_2}^* E_{\xi_2} - Q [2]_q^2 q^{-4} E_{\xi_3}^* E_{\xi_3}, \quad
E_{\xi_3} E_{\xi_3}^* \sim q^{-4} E_{\xi_3}^* E_{\xi_3}, \\
E_{\xi_1} E_{\xi_2}^* \sim q^{-2} E_{\xi_2}^* E_{\xi_1} - Q [2]_q q^{-2} E_{\xi_3}^* E_{\xi_2}, \quad
E_{\xi_1} E_{\xi_3}^* \sim E_{\xi_3}^* E_{\xi_1}, \quad
E_{\xi_2} E_{\xi_3}^* \sim q^{-2} E_{\xi_3}^* E_{\xi_2}.
\end{gathered}
\]
\end{lemma}

\begin{proof}
Since $\xi_3 = \alpha_2$, it is easy to show that $E_{\xi_i} E_{\xi_3}^* \sim q^{-(\xi_i, \xi_3)} E_{\xi_3}^* E_{\xi_i}$, see for instance \cite[Lemma 2]{mat-par}. To obtain the other relations we use the adjoint action of $\Uql$. We have
\begin{equation}
\label{eq:actf-xy}
\begin{split}
F_i \triangleright (X Y^*) & = (F_i \triangleright X) (K_i \triangleright Y)^* - X (E_i \triangleright Y)^*, \\
F_i \triangleright (Y^* X) & = - (E_i \triangleright Y)^* (K_i^{-1} \triangleright X) + Y^* (F_i \triangleright X),
\end{split}
\end{equation}
where we have used $X \triangleright Y^* = (S(X)^* \triangleright Y)^*$ and $S(F_i)^* = - E_i$. From \eqref{eq:actf-xy} we get
\[
\begin{split}
F_1 \triangleright (E_{\xi_i} E_{\xi_3}^*) & = q^{-2} (F_1 \triangleright E_{\xi_i}) E_{\xi_3}^* - E_{\xi_i} E_{\xi_2}^*, \\
F_1 \triangleright (E_{\xi_3}^* E_{\xi_i}) & = - q^{-(\alpha_1, \xi_i)} E_{\xi_2}^* E_{\xi_i} + E_{\xi_3}^* (F_1 \triangleright E_{\xi_i}).
\end{split}
\]
Applying $F_1$ to $E_{\xi_i} E_{\xi_3}^* \sim q^{-(\alpha_2, \xi_i)} E_{\xi_3}^* E_{\xi_i}$ we get, after some simple manipulations,
\[
E_{\xi_i} E_{\xi_2}^* - q^{-(\xi_2, \xi_i)} E_{\xi_2}^* E_{\xi_i} \sim q^{-2} (F_1 \triangleright E_{\xi_i}) E_{\xi_3}^*  - q^{-(\alpha_2, \xi_i)} E_{\xi_3}^* (F_1 \triangleright E_{\xi_i}).
\]
Consider the case $i = 1$. Since $(\xi_2, \xi_1) = 2$ we get
\[
E_{\xi_1} E_{\xi_2}^* - q^{-2} E_{\xi_2}^* E_{\xi_1} \sim q^{-2} E_{\xi_2} E_{\xi_3}^* - E_{\xi_3}^* E_{\xi_2} \sim (q^{-4} - 1) E_{\xi_3}^* E_{\xi_2}.
\]
Observe that $q^{-4} - 1 = - Q [2]_q q^{-2}$. Next consider $i = 2$. Since $(\xi_2, \xi_2) = 2$ we get
\[
E_{\xi_2} E_{\xi_2}^* - q^{-2} E_{\xi_2}^* E_{\xi_2} \sim [2]_q q^{-2} E_{\xi_3} E_{\xi_3}^* - [2]_q q^{-2} E_{\xi_3}^* E_{\xi_3} \sim [2]_q q^{-2} (q^{-4} - 1) E_{\xi_3}^* E_{\xi_3}.
\]

We are left with the relation for $E_{\xi_1} E_{\xi_1}^*$. From the formulae in \eqref{eq:actf-xy} we obtain
\[
\begin{split}
F_1 \triangleright (E_{\xi_1} E_{\xi_2}^*) & = E_{\xi_2} E_{\xi_2}^* - [2]_q E_{\xi_1} E_{\xi_1}^*, \\
F_1 \triangleright (E_{\xi_2}^* E_{\xi_1}) & = - [2]_q q^{-2} E_{\xi_1}^* E_{\xi_1} + E_{\xi_2}^* E_{\xi_2}, \\
F_1 \triangleright (E_{\xi_3}^* E_{\xi_2}) & = - E_{\xi_2}^* E_{\xi_2} + [2]_q E_{\xi_3}^* E_{\xi_3}.
\end{split}
\]
Using these expressions we compute
\[
\begin{split}
0 & \sim F_1 \triangleright (E_{\xi_1} E_{\xi_2}^* - q^{-2} E_{\xi_2}^* E_{\xi_1} + Q [2]_q q^{-2} E_{\xi_3}^* E_{\xi_2}) \\
& = (E_{\xi_2} E_{\xi_2}^* - q^{-2} E_{\xi_2}^* E_{\xi_2}) - [2]_q (E_{\xi_1} E_{\xi_1}^* - q^{-4} E_{\xi_1}^* E_{\xi_1}) - Q [2]_q q^{-2} E_{\xi_2}^* E_{\xi_2} + Q [2]_q^2 q^{-2} E_{\xi_3}^* E_{\xi_3} \\
& \sim - Q [2]_q^2 q^{-4} E_{\xi_3}^* E_{\xi_3} - [2]_q (E_{\xi_1} E_{\xi_1}^* - q^{-4} E_{\xi_1}^* E_{\xi_1}) - Q [2]_q q^{-2} E_{\xi_2}^* E_{\xi_2} + Q [2]_q^2 q^{-2} E_{\xi_3}^* E_{\xi_3}.
\end{split}
\]
Finally this can be rewritten as
\[
E_{\xi_1} E_{\xi_1}^* \sim q^{-4} E_{\xi_1}^* E_{\xi_1} - Q q^{-2} E_{\xi_2}^* E_{\xi_2} + Q^2 [2]_q q^{-3} E_{\xi_3}^* E_{\xi_3}. \qedhere
\]
\end{proof}

\subsection{Computing $D^2$}

We are now ready to compute the square of the Dolbeault--Dirac element $D = \dolb + \dolb^*$, as given in \cref{def:dolb-dirac}. The result will be an expression of the form $D^2 \sim \sum_{i, j = 1}^3 E_{\xi_i}^* E_{\xi_j} \otimes \Gamma_{i j}$, where $\Gamma_{i j}$ are some explicit elements of $\mathrm{End}(\Lq)$.

\begin{proposition}
\label{prop:D-squared}
We have $D^2 \sim \sum_{i, j = 1}^3 E_{\xi_i}^* E_{\xi_j} \otimes \Gamma_{i j}$, where
\[
\begin{split}
\Gamma_{1 1} & := \gamma(y_1)^* \gamma(y_1) + q^{-4} \gamma(y_1) \gamma(y_1)^*, \\
\Gamma_{2 2} & := \gamma(y_2)^* \gamma(y_2) + q^{-2} \gamma(y_2) \gamma(y_2)^* - \qdiff q^{-2} \gamma(y_1) \gamma(y_1)^*, \\
\Gamma_{3 3} & := \gamma(y_3)^* \gamma(y_3) + q^{-4} \gamma(y_3) \gamma(y_3)^* + \qdiff^2 [2]_q q^{-3} \gamma(y_1) \gamma(y_1)^* - \qdiff [2]_q^2 q^{-4} \gamma(y_2) \gamma(y_2)^*,
\end{split}
\]
while in the case $i < j$ we have
\[
\begin{split}
\Gamma_{1 2} & := \gamma(y_1)^* \gamma(y_2) + q^{-2} \gamma(y_2) \gamma(y_1)^*, \\
\Gamma_{1 3} & := \gamma(y_1)^* \gamma(y_3) + \gamma(y_3) \gamma(y_1)^*, \\
\Gamma_{2 3} & := \gamma(y_2)^* \gamma(y_3) + q^{-2} \gamma(y_3) \gamma(y_2)^* - \qdiff [2]_q q^{-2} \gamma(y_2) \gamma(y_1)^*,
\end{split}
\]
and finally in the case $i > j$ we have $\Gamma_{i j} := \Gamma_{j i}^*$.
\end{proposition}

\begin{proof}
We have $D^2 = \dolb \dolb^* + \dolb^* \dolb$, since $\dolb^2 = 0$ and $(\dolb^*)^2 = 0$. Then we can write
\[
D^2 = \sum_{i, j = 1}^3 E_{\xi_i} E_{\xi_j}^* \otimes \gamma(y_i) \gamma(y_i)^* + \sum_{i, j = 1}^3 E_{\xi_j}^* E_{\xi_i} \otimes \gamma(y_i)^* \gamma(y_i).
\]
Next we use the commutation relations from \cref{lem:rel-xi-xis}. We obtain
\[
\begin{split}
D^2 & \sim E_{\xi_1}^* E_{\xi_1} \otimes \left( \gamma(y_1)^* \gamma(y_1) + q^{-4} \gamma(y_1) \gamma(y_1)^* \right) \\
& + E_{\xi_2}^* E_{\xi_2} \otimes \left( \gamma(y_2)^* \gamma(y_2) + q^{-2} \gamma(y_2) \gamma(y_2)^* - \qdiff q^{-2} \gamma(y_1) \gamma(y_1)^* \right) \\
& + E_{\xi_3}^* E_{\xi_3} \otimes \left( \gamma(y_3)^* \gamma(y_3) + q^{-4} \gamma(y_3 )\gamma(y_3)^* + \qdiff^2 [2]_q q^{-3} \gamma(y_1) \gamma(y_1)^* - \qdiff [2]_q^2 q^{-4} \gamma(y_2) \gamma(y_2)^* \right) \\
& + E_{\xi_1}^* E_{\xi_2} \otimes \left( \gamma(y_1)^* \gamma(y_2) + q^{-2} \gamma(y_2) \gamma(y_1)^* \right) + E_{\xi_1}^* E_{\xi_3} \otimes \left( \gamma(y_1)^* \gamma(y_3) + \gamma(y_3) \gamma(y_1)^* \right) \\
& + E_{\xi_2}^* E_{\xi_3} \otimes \left( \gamma(y_2)^* \gamma(y_3) + q^{-2} \gamma(y_3) \gamma(y_2)^* - \qdiff [2]_q q^{-2} \gamma(y_2) \gamma(y_1)^* \right) \\
& + E_{\xi_2}^* E_{\xi_1} \otimes \left( \gamma(y_2)^* \gamma(y_1) + q^{-2} \gamma(y_1) \gamma(y_2)^* \right) + E_{\xi_3}^* E_{\xi_1} \otimes \left(\gamma(y_3)^* \gamma(y_1) + \gamma(y_1) \gamma(y_3)^* \right) \\
& + E_{\xi_3}^* E_{\xi_2} \otimes \left( \gamma(y_3)^* \gamma(y_2) + q^{-2} \gamma(y_2) \gamma(y_3)^* - \qdiff [2]_q q^{-2} \gamma(y_1) \gamma(y_2)^* \right).
\end{split}
\]
Then the claim follows immediately from this expression.
\end{proof}

As mentioned in the introduction, our aim is to obtain a quantum analogue of the Parthasarathy formula for $D^2$. To proceed we need to an appropriate central element of $U_q(\mathfrak{sp}_4)$, which will play the role of the quadratic Casimir. This will be constructed in the next section. Once this is done, our next task will be to compare $D^2$ with this central element.

\section{Derivation Casimir}
\label{sec:casimir}

In this section we construct a central element of $U_q(\mathfrak{sp}_4)$, which will play the role of the quadratic Casimir.
It will be derived from a general construction which uses the R-matrix.

\subsection{Central elements}

Recall that the R-matrix of $\Uqg$, which we denote by $\Rmat$, is an element in a completion of $\Uqg \otimes \Uqg$ which among other properties satisfies the identity $\Rmat \Delta(X) = \Delta^{\mathrm{op}}(X) \Rmat$ for all $X \in \Uqg$.
We will now outline a general costruction of central elements of $\Uqg$, which is based on the properties of the element $\Rmat$.

First recall that the R-matrix can be written in the form $\Rmat = \widetilde{\Rmat} \cdot \kappa^-$, where $\widetilde{\Rmat}$ is the quasi-R-matrix and $\kappa^-$ is defined by $\kappa^- (v_\mu \otimes v_\nu) = q^{-(\mu, \nu)} v_\mu \otimes v_\nu$. An explicit expression for $\widetilde{\Rmat}$ is given in \cite[Section 8.30]{jantzen} (note that this reference uses opposite conventions for the R-matrix). Given the reduced expression $w_0 = s_{i_1} \cdots s_{i_n}$, we have
\[
\widetilde{\Rmat} = \widetilde{\Rmat}^{[n]} \cdot \widetilde{\Rmat}^{[n - 1]} \cdots \widetilde{\Rmat}^{[2]} \cdot \widetilde{\Rmat}^{[1]},
\]
where $\widetilde{\Rmat}^{[j]}$ is the quasi-R-matrix of $U_{q_{i_j}}(\mathfrak{sl}_2)$ given by
\[
\widetilde{\Rmat}^{[j]} = \sum_{r = 0}^\infty (-1)^r q_{i_j}^{- r (r - 1) / 2} \frac{(q_{i_j} - q_{i_j}^{-1})^r}{[r]_{i_j}!} E_{\beta_i}^r \otimes F_{\beta_i}^r.
\]
We will write schematically $\Rmat = \sum_A c_A (E^A \otimes F^A) \kappa^-$, where $A = (a_1, \cdots, a_n)$ and
\[
E^A := E_{\beta_n}^{a_n} \cdots E_{\beta_1}^{a_1}, \quad
F^A := F_{\beta_n}^{a_n} \cdots F_{\beta_1}^{a_1}.
\]
The coefficients $c_A$ are defined implicitly by the previous formula.
It is also known that $\Rmat^* = \Rmat_{21}$, where the latter is the flipped R-matrix. Finally consider the map
\[
I : \Cqg \to \Uqg, \quad \omega \mapsto (\id \otimes \omega) (\Rmat^* \Rmat).
\]

\begin{lemma}
Let $\eta_v \in \Cqg$ be defined by $\eta_v(X) = (v, X v)$, with $\mathrm{wt}(v) = \lambda$. Then
\[
I(\eta_v) = \sum_{A, B} c_A c_B (F^A v, F^B v) E^{A *} E^B K_{2 \lambda}^{-1},
\]
where in the sum we must have $\mathrm{wt}(A) = \mathrm{wt}(B)$.
\end{lemma}

\begin{proof}
It follows from the definitions and the observation that $\kappa^- (v_\mu \otimes v_\nu) = K_\nu^{-1} v_\mu \otimes v_\nu$.
\end{proof}

It is well-known that one can use the map $I$ to obtain central elements of $\Uqg$, see for instance \cite[Section 3]{LiGo92} and \cite[Section 1.2]{Bau98}.
This construction makes use of the quantum trace $\tau_{q, V}$, which is defined by $\tau_{q, V}(X) = \mathrm{Tr}(K_{2 \rho}^{-1} X)$ for a simple $\Uqg$-module $V$. Here the element $K_{2 \rho}^{-1}$ is related to the square of the antipode, which in our conventions is given by $S^2(X) = K_{2 \rho}^{-1} X K_{2 \rho}$. Then $I(\tau_{q, V})$ turns out to be a central element.

We will now obtain a concrete expression for this element.

\begin{proposition}
\label{prop:cas-general}
Let $V$ be a simple $\Uqg$-module with an orthonormal basis $\{ v_j \}_{j \in J}$ of weight $\mathrm{wt}(v_j) = \lambda_j$. Then we have the central element
\[
I(\tau_{q, V}) = \sum_{j \in J} \sum_{A, B} c_A c_B q^{-(2 \rho, \lambda_j)} (F^A v_j, F^B v_j) E^{A *} E^B K_{2 \lambda_j}^{-1}.
\]
\end{proposition}

\begin{proof}
It follows from the fact that $\tau_{q, V} (X) = \sum_{j \in J} q^{-(2 \rho, \lambda_j)} \eta_{v_j} (X)$.
\end{proof}

We will rescale this central element for notational convenience.
\begin{definition}
We define a central element $\Cas \in \Uqg$ by
\[
\Cas := \frac{(\id \otimes \tau_{q, V}) (\Rmat^* \Rmat)}{(q - q^{-1})^2}.
\]
\end{definition}

Upon adding an appropriate multiple of the identity to $\Cas$, for instance $\tau_{q, V} (1) / (q - q^{-1})^2$, it is possible to recover the usual quadratic Casimir in the limit $q \to 1$. However this will not be very important for us, hence we will omit this constant.

\begin{corollary}
\label{cor:value-casimir}
Let $V(\Lambda)$ be a simple $\Uqg$-module of highest weight $\Lambda$. Then $\Cas = c_\Lambda \id$ on $V(\Lambda)$, where $c_\Lambda = \sum_{j \in J} q^{-2 (\lambda_j, \Lambda + \rho)} / (q - q^{-1})^2$.
\end{corollary}

\begin{proof}
Since $\Cas$ is central, it suffices to compute its action on a highest weight vector $v_\Lambda$. We have $E^B v_\Lambda = 0$ unless $B = 0$, which in turn implies $A = 0$ by weight reasons. Then
\[
\Cas v_\Lambda = \sum_{j \in J} \frac{q^{-(2 \rho, \lambda_j)}}{(q - q^{-1})^2} K_{2 \lambda_j}^{-1} v_\Lambda = \sum_{j \in J} \frac{q^{-2 (\lambda_j, \Lambda + \rho)}}{(q - q^{-1})^2} v_\Lambda. \qedhere
\]
\end{proof}

\subsection{Explicit formula}

In this subsection we will simplify the general expression given in \cref{prop:cas-general}, in the case where $\mathfrak{g} = \mathfrak{sp}_4$ and $V$ is the fundamental representation.

We begin by observing the vanishing of some terms in the given representation.

\begin{lemma}
\label{lem:f-vanish}
We have $\pi(F_{\beta_i}^2) = 0$ for $i = 1, 2, 3, 4$ and
\[
\begin{gathered}
\pi(F_{\beta_2} F_{\beta_1}) = \pi(F_{\beta_3} F_{\beta_2}) = \pi(F_{\beta_4} F_{\beta_2}) = \pi(F_{\beta_4} F_{\beta_3}) = 0, \\
\pi(F_{\beta_3} F_{\beta_2} F_{\beta_1}) = \pi(F_{\beta_4} F_{\beta_2} F_{\beta_1}) = \pi(F_{\beta_4} F_{\beta_3} F_{\beta_2}) = 0.
\end{gathered}
\]
On the other hand $\pi(F_{\beta_3} F_{\beta_1})$ and $\pi(F_{\beta_4} F_{\beta_1})$ are non-zero.
\end{lemma}

\begin{proof}
Follows easily from the formulae given in \cref{lem:fundamental-c2}.
\end{proof}

We are now ready to derive the main result of this section.

\begin{proposition}
\label{prop:casimir-rmatrix}
We have the central element $\Cas = \Cas_c + \Cas_q$, where
\[
\begin{split}
\Cas_c & = (q^{-4} K_{2 \lambda_1}^{-1} + q^{-2} K_{2 \lambda_2}^{-1} + q^2 K_{2 \lambda_3}^{-1} + q^4 K_{2 \lambda_4}^{-1}) / (q - q^{-1})^2 \\
& + E_{\beta_1}^* E_{\beta_1} (q^{-5} K_{2 \lambda_1}^{-1} + q K_{2 \lambda_3}^{-1}) + [2]_q^2 E_{\beta_2}^* E_{\beta_2} q^{-6} K_{2 \lambda_1}^{-1} \\
& + E_{\beta_3}^* E_{\beta_3} (q^{-7} K_{2 \lambda_1}^{-1} + q^{-1} K_{2 \lambda_2}^{-1}) + [2]_q^2 E_{\beta_4}^* E_{\beta_4} q^{-4} K_{2 \lambda_2}^{-1},
\end{split}
\]
and the quantum part is given by
\[
\begin{split}
\Cas_q & = - (q - q^{-1}) [2]_q q^{-5} (E_{\beta_1}^* E_{\beta_3}^* E_{\beta_2} + E_{\beta_2}^* E_{\beta_3} E_{\beta_1}) K_{2 \lambda_1}^{-1} \\
& - (q - q^{-1}) [2]_q q^{-7} (E_{\beta_1}^* E_{\beta_4}^* E_{\beta_3} + E_{\beta_3}^* E_{\beta_4} E_{\beta_1}) K_{2 \lambda_1}^{-1} \\
& + (q - q^{-1})^2 q^{-4} E_{\beta_1}^* E_{\beta_3}^* E_{\beta_3} E_{\beta_1} K_{2 \lambda_1}^{-1} \\
& + (q - q^{-1})^2 [2]_q^2 q^{-7} E_{\beta_1}^* E_{\beta_4}^* E_{\beta_4} E_{\beta_1} K_{2 \lambda_1}^{-1}.
\end{split}
\]
\end{proposition}

\begin{proof}
Consider the general expression given in \cref{prop:cas-general}, up to rescaling.
Then we can write $\Cas = \sum_{m, n} \Cas_{m, n}$, where $\Cas_{m, n}$ contains terms with $m = |A|$ and $n = |B|$.
To simplify the expression we consider the inner products $(F^A v_j, F^B v_j)$.
It follows from \cref{lem:f-vanish} that $\pi(F^A) = 0$ for $|A| > 2$, hence it suffices to consider $0 \leq m, n \leq 2$.

Consider the case $m = n = 0$. We have $(v_i, v_i) = 1$ and $c_0 = 1$. Then
\[
\Cas_{0, 0} = q^{- (2\rho, \lambda_1)} K_{2 \lambda_1}^{-1} + q^{- (2\rho, \lambda_2)} K_{2 \lambda_2}^{-1} + q^{- (2\rho, \lambda_3)} K_{2 \lambda_3}^{-1} + q^{- (2\rho, \lambda_4)} K_{2 \lambda_4}^{-1}.
\]
Next consider $m = 1$ and $n = 0$. We have $(F_{\beta_i} v_j, v_j) = 0$ by weight reasons. Similarly for the case $m = 0$ and $n = 1$. Hence $\Cas_{1, 0} = \Cas_{0, 1} = 0$.

Now consider the case $m = n = 1$. If $i \neq j$ then $(F_{\beta_i} v_k, F_{\beta_j} v_k) = 0$ by weight reasons. In the case $i = j$ we have the non-zero inner products
\[
\begin{gathered}
(F_{\beta_1} v_1, F_{\beta_1} v_1) = q^{-1}, \quad (F_{\beta_1} v_3, F_{\beta_1} v_3) = q^{-1}, \quad (F_{\beta_2} v_1, F_{\beta_2} v_1) = q^{-2}, \\
(F_{\beta_3} v_1, F_{\beta_3} v_1) = q^{-3}, \quad (F_{\beta_3} v_2, F_{\beta_3} v_2) = q, \quad (F_{\beta_4} v_2, F_{\beta_4} v_2) = q^{-2}.
\end{gathered}
\]
Hence we get the terms
\[
\begin{split}
\Cas_{1, 1} & = c_{\beta_1}^2 E_{\beta_1}^* E_{\beta_1} (q^{-(2\rho, \lambda_1)} q^{-1} K_{2 \lambda_1}^{-1} + q^{-(2\rho, \lambda_3)} q^{-1} K_{2 \lambda_3}^{-1}) + c_{\beta_2}^2 E_{\beta_2}^* E_{\beta_2} q^{-(2\rho, \lambda_1)} q^{-2} K_{2 \lambda_1}^{-1} \\
& + c_{\beta_3}^2 E_{\beta_3}^* E_{\beta_3} (q^{-(2\rho, \lambda_1)} q^{-3} K_{2 \lambda_1}^{-1} + q^{-(2\rho, \lambda_2)} q K_{2 \lambda_2}^{-1}) + c_{\beta_4}^2 E_{\beta_4}^* E_{\beta_4} q^{-(2\rho, \lambda_2)} q^{-2} K_{2 \lambda_2}^{-1}.
\end{split}
\]

Next consider $m = 2$ and $n = 1$. The only two possibilities for $F^A$ are $F_{\beta_3} F_{\beta_1}$ and $F_{\beta_4} F_{\beta_1}$. Then by weight reasons the only possibly non-zero inner products are $(F_{\beta_3} F_{\beta_1} v_j, F_{\beta_2} v_j)$ and $(F_{\beta_4} F_{\beta_1} v_j, F_{\beta_3} v_j)$. We have in particular
\[
(F_{\beta_3} F_{\beta_1} v_1, F_{\beta_2} v_1) = q^{-1}, \quad (F_{\beta_4} F_{\beta_1} v_1, F_{\beta_3} v_1) = q^{-3}.
\]
Therefore we get
\[
\Cas_{2, 1} = c_{\beta_1 + \beta_3} c_{\beta_2} q^{-(2\rho, \lambda_1)} E_{\beta_1}^* E_{\beta_3}^* E_{\beta_2} q^{-1} K_{2 \lambda_1}^{-1} + c_{\beta_1 + \beta_4} c_{\beta_3} q^{-(2\rho, \lambda_1)} E_{\beta_1}^* E_{\beta_4}^* E_{\beta_3} q^{-3} K_{2 \lambda_1}^{-1}.
\]
The case $m = 1$ and $n = 2$ follows immediately from the fact that $(F^A v_i, F^B v_i) = (F^B v_i, F^A v_i)$, as all the matrix entries are real, as can be seen from \cref{lem:fundamental-c2}. Then we get
\[
\Cas_{1, 2} = c_{\beta_1 + \beta_3} c_{\beta_2} q^{-(2\rho, \lambda_1)} E_{\beta_2}^* E_{\beta_3} E_{\beta_1} q^{-1} K_{2 \lambda_1}^{-1} + c_{\beta_1 + \beta_4} c_{\beta_3} q^{-(2\rho, \lambda_1)} E_{\beta_3}^* E_{\beta_4} E_{\beta_1} q^{-3} K_{2 \lambda_1}^{-1}.
\]
Finally we have the case $m = n = 2$. By weight reasons the only possibly non-zero terms are $(F_{\beta_3} F_{\beta_1} v_i, F_{\beta_3} F_{\beta_1} v_i)$ and $(F_{\beta_4} F_{\beta_1} v_i, F_{\beta_4} F_{\beta_1} v_i)$. The non-zero inner products are
\[
(F_{\beta_3} F_{\beta_1} v_1, F_{\beta_3} F_{\beta_1} v_1) = 1, \quad (F_{\beta_4} F_{\beta_1} v_1, F_{\beta_4} F_{\beta_1} v_1) = q^{-3}.
\]
Therefore we get the terms
\[
\Cas_{2, 2} = c_{\beta_1 + \beta_3}^2 E_{\beta_1}^* E_{\beta_3}^* E_{\beta_3} E_{\beta_1} q^{-(2\rho, \lambda_1)} K_{2 \lambda_1}^{-1} + c_{\beta_1 + \beta_4}^2 E_{\beta_1}^* E_{\beta_4}^* E_{\beta_4} E_{\beta_1} q^{-(2\rho, \lambda_1)} q^{-3} K_{2 \lambda_1}^{-1}.
\]

Finally we have to plug the explicit values into the previous formulae. We have
\[
(2\rho, \lambda_1) = 4, \quad (2\rho, \lambda_2) = 2, \quad (2\rho, \lambda_3) = -2, \quad (2\rho, \lambda_4) = -4.
\]
For the coefficients coming from the R-matrix we have
\[
\begin{gathered}
c_{\beta_1} = c_{\beta_3} = - (q - q^{-1}) = - Q, \quad
c_{\beta_2} = c_{\beta_4} = - (q^2 - q^{-2}) = - Q [2]_q, \\
c_{\beta_1 + \beta_3} = c_{\beta_1} c_{\beta_3} = Q^2 [2]_q, \quad 
c_{\beta_1 + \beta_4} = c_{\beta_1} c_{\beta_4} = Q^2 [2]_q.
\end{gathered}
\]
Plugging in all these values and collecting the various terms we obtain the result.
\end{proof}

\subsection{The classical limit}

We take as a basis of $\mathfrak{g} = \mathfrak{sp}_4$ the elements $H_1$, $H_2$ and the classical limit of the quantum root vectors $E_{\beta_i}$ and $F_{\beta_i}$ given in \cref{lem:root-E}. Then with respect to a rescaling of the Killing form we have the dual basis
\[
\begin{gathered}
\widetilde{H_1} = H_1 + H_2, \quad \widetilde{H_2} = H_1 + 2 H_2, \quad \widetilde{E_{\beta_1}} = F_{\beta_1}, \quad \widetilde{E_{\beta_2}} = 2 F_{\beta_2}, \\
\widetilde{E_{\beta_3}} = F_{\beta_3}, \quad \widetilde{E_{\beta_4}} = 2 F_{\beta_4}, \quad \widetilde{F_{\beta_1}} = E_{\beta_1}, \quad \widetilde{F_{\beta_2}} = 2 E_{\beta_2}, \quad \widetilde{F_{\beta_3}} = E_{\beta_3}, \quad \widetilde{F_{\beta_4}} = 2 E_{\beta_4}.
\end{gathered}
\]
In general the quadratic Casimir is given by $\mathcal{C} = \sum_i X_i \widetilde{X_i}$. In this case we have
\[
\begin{split}
\mathcal{C} & = H_1^2 + 2 H_1 H_2 + 2 H_2^2 + E_{\beta_1} F_{\beta_1} + 2 E_{\beta_2} F_{\beta_2} + E_{\beta_3} F_{\beta_3} + 2 E_{\beta_4} F_{\beta_4} \\
& + F_{\beta_1} E_{\beta_1} + 2 F_{\beta_2} E_{\beta_2} + F_{\beta_3} E_{\beta_3} + 2 F_{\beta_4} E_{\beta_4}.
\end{split}
\]
This can be rewritten using the commutation relations
\[
[E_{\beta_1}, F_{\beta_1}] = H_1, \quad [E_{\beta_2}, F_{\beta_2}] = H_1 + H_2, \quad [E_{\beta_3}, F_{\beta_3}] = H_1 + 2 H_2, \quad [E_{\beta_4}, F_{\beta_4}] = H_2.
\]
Therefore we obtain the following expression for the quadratic Casimir
\[
\mathcal{C} = H_1^2 + 2 H_1 H_2 + 2 H_2^2 + 4 H_1 + 6 H_2 + 2 F_{\beta_1} E_{\beta_1} + 4 F_{\beta_2} E_{\beta_2} + 2 F_{\beta_3} E_{\beta_3} + 4 F_{\beta_4} E_{\beta_4}.
\]
Moreover we have $E_{\beta_i}^* = F_{\beta_i}$. Then it is easy to see that, upon the addition of an appropriate constant, this coincides with the limit $q \to 1$ of the central element of \cref{prop:casimir-rmatrix}.

\section{Rewriting the Casimir}
\label{sec:rewriting}

In this section we will obtain a different expression for the Casimir element $\Cas$, which will facilitate the comparison with the Dolbeault--Dirac element $D$. Along the way we introduce a certain algebra $\algspace$, which allows to formulate the comparison problem algebraically.

\subsection{Commutation relations}

We start by deriving commutation relations between the element $E_{\beta_1}^*$ and the quantum root vectors $E_{\beta_i}$. We will use the relations
\[
E_{\beta_1} E_{\beta_3} - E_{\beta_3} E_{\beta_1} = [2]_q E_{\beta_2}, \quad
E_{\beta_1} E_{\beta_4} - q^{-2} E_{\beta_4} E_{\beta_1} = E_{\beta_3}.
\]
which have been derived in \cref{lem:levi-up}.

\begin{lemma}
\label{lem:rel-e-es}
We have the commutation relations
\[
\begin{gathered}
E_{\beta_1}^* E_{\beta_1} = q^2 E_{\beta_1} E_{\beta_1}^* - q^2 \frac{K_1^2 - 1}{q - q^{-1}}, \quad
E_{\beta_1}^* E_{\beta_2} = q^2 E_{\beta_2} E_{\beta_1}^* + q^2 E_{\beta_3}, \\
E_{\beta_1}^* E_{\beta_3} = E_{\beta_3} E_{\beta_1}^* + [2]_q E_{\beta_4}, \quad
E_{\beta_1}^* E_{\beta_4} = q^{-2} E_{\beta_4} E_{\beta_1}^*.
\end{gathered}
\]
\end{lemma}

\begin{proof}
The relations for $E_{\beta_1}^* E_{\beta_1}$, $E_{\beta_1}^* E_{\beta_3}$ and $E_{\beta_1}^* E_{\beta_4}$ follow from simple computations that we omit.
We show the more complicated identity for $E_{\beta_1}^* E_{\beta_2}$. We can write
\[
E_{\beta_1}^* E_{\beta_2} = \frac{1}{[2]_q} \left( E_{\beta_1}^* E_{\beta_1} E_{\beta_3} - E_{\beta_1}^* E_{\beta_3} E_{\beta_1} \right).
\]
Using the other commutation relations we obtain
\[
\begin{split}
E_{\beta_1}^* E_{\beta_2} & = \frac{1}{[2]_q} \left( q^2 E_{\beta_1} E_{\beta_1}^* E_{\beta_3} - q^2 \frac{K_1^2 - 1}{q - q^{-1}} E_{\beta_3} - E_{\beta_3} E_{\beta_1}^* E_{\beta_1} - [2]_q E_{\beta_4} E_{\beta_1} \right) \\
& = \frac{1}{[2]_q} \left( q^2 E_{\beta_1} E_{\beta_3} E_{\beta_1}^* + [2]_q q^2 E_{\beta_1} E_{\beta_4} - q^2 \frac{K_1^2 - 1}{q - q^{-1}} E_{\beta_3} \right) \\
& + \frac{1}{[2]_q} \left( -q^2 E_{\beta_3} E_{\beta_1} E_{\beta_1}^* + q^2 E_{\beta_3} \frac{K_1^2 - 1}{q - q^{-1}} - [2]_q E_{\beta_4} E_{\beta_1} \right).
\end{split}
\]
From $E_{\beta_1} E_{\beta_3} - E_{\beta_3} E_{\beta_1} = [2]_q E_{\beta_2}$ and $K_1 E_{\beta_3} = E_{\beta_3} K_1$ we get
\[
E_{\beta_1}^* E_{\beta_2} = q^2 E_{\beta_2} E_{\beta_1}^* + q^2 E_{\beta_1} E_{\beta_4} - E_{\beta_4} E_{\beta_1} = q^2 E_{\beta_2} E_{\beta_1}^* + q^2 E_{\beta_3},
\]
where in the last step we have used $E_{\beta_1} E_{\beta_4} - q^{-2} E_{\beta_4} E_{\beta_1} = E_{\beta_3}$.
\end{proof}

\begin{lemma}
\label{lem:rel-rewrite-cas}
We have the following relations
\[
\begin{split}
(E_{\beta_3} E_{\beta_1})^* E_{\beta_2} & = q^2 E_{\beta_3}^* E_{\beta_2} E_{\beta_1}^* + q^2 E_{\beta_3}^* E_{\beta_3} - [2]_q E_{\beta_2}^* E_{\beta_2}, \\
(E_{\beta_4} E_{\beta_1})^* E_{\beta_3} & = q^2 E_{\beta_4}^* E_{\beta_3} E_{\beta_1}^* + [2]_q q^2 E_{\beta_4}^* E_{\beta_4} - q^2 E_{\beta_3}^* E_{\beta_3}, \\
(E_{\beta_3} E_{\beta_1})^* E_{\beta_3} E_{\beta_1} & = E_{\beta_3}^* E_{\beta_3} E_{\beta_1}^* E_{\beta_1} + [2]_q (E_{\beta_3}^* E_{\beta_4} E_{\beta_1} - E_{\beta_2}^* E_{\beta_3} E_{\beta_1}), \\
(E_{\beta_4} E_{\beta_1})^* E_{\beta_4} E_{\beta_1} & = E_{\beta_4}^* E_{\beta_4} E_{\beta_1}^* E_{\beta_1} - q^2 E_{\beta_3}^* E_{\beta_4} E_{\beta_1}.
\end{split}
\]
\end{lemma}

\begin{proof}
These follow easily from the commutation relations of \cref{lem:rel-e-es}.
\end{proof}

\subsection{Moving Levi terms to the right}

Our immediate goal is to rewrite the Casimir $\Cas$ by bringing the elements $E_{\beta_1}$ and $E_{\beta_1}^*$ all the way to the right. Note that these elements belong to the quantized Levi factor $\Uql$. The reason for doing so will be clear once we include the Clifford algebra part, that is the algebra $\mathrm{End}(\Lq)$.

Notice that the classical part of the Casimir, namely $\Cas_c$ in the notation of \cref{prop:casimir-rmatrix}, is already of this form. For this reason we will focus on the quantum part $\Cas_q$.

\begin{lemma}
\label{lem:quantum-casimir}
We have the identity
\[
\begin{split}
\Cas_q & = \qdiff [2]_q^2 q^{-5} \left( E_{\beta_2}^* E_{\beta_2} - \qdiff [2]_q^{-1} q E_{\beta_3}^*E_{\beta_3} (1 - [2]_q^{-1} E_{\beta_1}^* E_{\beta_1}) - E_{\beta_4}^* E_{\beta_4} (1 - \qdiff q^{-2} E_{\beta_1}^* E_{\beta_1}) \right) K_{2 \lambda_1}^{-1} \\
& - \qdiff [2]_q q^{-5} \left( q^2 E_{\beta_2}^* E_{\beta_3} E_{\beta_1} + q^2 E_{\beta_3}^* E_{\beta_2} E_{\beta_1}^* + E_{\beta_3}^* E_{\beta_4} E_{\beta_1} + E_{\beta_4}^* E_{\beta_3} E_{\beta_1}^* \right) K_{2 \lambda_1}^{-1}.
\end{split}
\]
\end{lemma}

\begin{proof}
Let us begin by writing
\[
\Cas_q = - \qdiff [2]_q q^{-5} \Cas_q^\prime K_{2 \lambda_1}^{-1} + \qdiff^2 q^{-4} \Cas_q^{\prime \prime} K_{2 \lambda_1}^{-1},
\]
which amounts to the definitions
\[
\begin{split}
\Cas_q^\prime & = (E_{\beta_3} E_{\beta_1})^* E_{\beta_2} + q^{-2} (E_{\beta_4} E_{\beta_1})^* E_{\beta_3} + E_{\beta_2}^* E_{\beta_3} E_{\beta_1} + q^{-2} E_{\beta_3}^* E_{\beta_4} E_{\beta_1}, \\
\Cas_q^{\prime \prime} & = (E_{\beta_3} E_{\beta_1})^* E_{\beta_3} E_{\beta_1} + [2]_q^2 q^{-3} (E_{\beta_4} E_{\beta_1})^* E_{\beta_4} E_{\beta_1}.
\end{split}
\]
Using the commutation relations of \cref{lem:rel-rewrite-cas} we obtain
\[
\begin{split}
\Cas_q^\prime & = - [2]_q E_{\beta_2}^* E_{\beta_2} + \qdiff q E_{\beta_3}^*E_{\beta_3} + [2]_q E_{\beta_4}^* E_{\beta_4} \\
& + E_{\beta_2}^* E_{\beta_3} E_{\beta_1} + q^2 E_{\beta_3}^* E_{\beta_2} E_{\beta_1}^* + q^{-2} E_{\beta_3}^* E_{\beta_4} E_{\beta_1} + E_{\beta_4}^* E_{\beta_3} E_{\beta_1}^*.
\end{split}
\]
Similarly for the other term we get
\[
\begin{split}
\Cas_q^{\prime \prime} & = E_{\beta_3}^* E_{\beta_3} E_{\beta_1}^* E_{\beta_1} + [2]_q^2 q^{-3} E_{\beta_4}^* E_{\beta_4} E_{\beta_1}^* E_{\beta_1} \\
& - [2]_q E_{\beta_2}^* E_{\beta_3} E_{\beta_1} - [2]_q q^{-2} E_{\beta_3}^* E_{\beta_4} E_{\beta_1}.
\end{split}
\]
Putting all together we obtain
\[
\begin{split}
\frac{q^5}{\qdiff [2]_q} \Cas_q K_{2 \lambda_1} & = [2]_q E_{\beta_2}^* E_{\beta_2} - \qdiff q E_{\beta_3}^*E_{\beta_3} - [2]_q E_{\beta_4}^* E_{\beta_4} \\
& - E_{\beta_2}^* E_{\beta_3} E_{\beta_1} - q^2 E_{\beta_3}^* E_{\beta_2} E_{\beta_1}^* - q^{-2} E_{\beta_3}^* E_{\beta_4} E_{\beta_1} - E_{\beta_4}^* E_{\beta_3} E_{\beta_1}^* \\
& + \qdiff [2]_q^{-1} q E_{\beta_3}^* E_{\beta_3} E_{\beta_1}^* E_{\beta_1} + \qdiff [2]_q q^{-2} E_{\beta_4}^* E_{\beta_4} E_{\beta_1}^* E_{\beta_1} \\
& - \qdiff q E_{\beta_2}^* E_{\beta_3} E_{\beta_1} - \qdiff q^{-1} E_{\beta_3}^* E_{\beta_4} E_{\beta_1}.
\end{split}
\]
After some simplifications, this gives the result.
\end{proof}

It is now easy to obtain the expression for $\Cas$ we were looking for.

\begin{proposition}
\label{prop:cas-to-the-right}
The central element $\Cas$ can be written as
\[
\begin{split}
\Cas & = (q^{-4} K_{2 \lambda_1}^{-1} + q^{-2} K_{2 \lambda_2}^{-1} + q^2 K_{2 \lambda_3}^{-1} + q^4 K_{2 \lambda_4}^{-1}) / (q - q^{-1})^2 \\
& + E_{\beta_1}^* E_{\beta_1} (q^{-5} K_{2 \lambda_1}^{-1} + q K_{2 \lambda_3}^{-1}) + E_{\beta_2}^* E_{\beta_2} [2]_q^2 q^{-4} K_{2 \lambda_1}^{-1} \\
& + E_{\beta_3}^* E_{\beta_3} ((q^{-5} - Q q^{-2}) K_{2 \lambda_1}^{-1} + q^{-1} K_{2 \lambda_2}^{-1} + \qdiff^2 q^{-4} E_{\beta_1}^* E_{\beta_1} K_{2 \lambda_1}^{-1}) \\
& + E_{\beta_4}^* E_{\beta_4} [2]_q^2 (q^{-4} K_{2 \lambda_2}^{-1} - \qdiff q^{-5} K_{2 \lambda_1}^{-1} + \qdiff^2 q^{-7} E_{\beta_1}^* E_{\beta_1} K_{2 \lambda_1}^{-1}) \\
& - Q [2]_q q^{-5} (q^2 E_{\beta_2}^* E_{\beta_3} E_{\beta_1} + q^2 E_{\beta_3}^* E_{\beta_2} E_{\beta_1}^* + E_{\beta_3}^* E_{\beta_4} E_{\beta_1} + E_{\beta_4}^* E_{\beta_3} E_{\beta_1}^*) K_{2 \lambda_1}^{-1}.
\end{split}
\]
\end{proposition}

\begin{proof}
This follows from \cref{prop:casimir-rmatrix} and \cref{lem:quantum-casimir}.
\end{proof}

\subsection{Embedding the Casimir}

So far we have only considered $\Cas$ as an element of $\Uqg$, but our goal will be to compare it with $D^2$ acting on $\spinbun$. Clearly we can embed $\Cas$ into $\Uqg \otimes \mathrm{End}(\Lq)$ by $\Cas \otimes 1$, but we also want to take into account the relations of $\spinbun$. This will be done by introducing an appropriate quotient of $\Uqg \otimes \mathrm{End}(\Lq)$.

First we fix some notation: let $X \in \Uqg$, $T \in \mathrm{End}(\Lq)$ and $Y \in \Uql$.
We consider $\Uqg$ as a right $\Uql$-module by right multiplication, which we write as $X \cdot Y = X Y$. On the other hand we consider $\mathrm{End}(\Lq)$ as a left $\Uql$-module by $Y \cdot T = T \circ S(Y)$, where on the right-hand side we have composition of operators.

\begin{definition}
With the above notation, we can form the tensor product of $\Uql$-modules
\[
\algspace := \Uqg \otimes_{\Uql} \mathrm{End}(\Lq).
\]
\end{definition}

Let us explain the motivation for this definition. With our definition of the actions of $\Uql$, the identity $X \cdot Y \otimes_{\Uql} T = X \otimes_{\Uql} Y \cdot T$ in $\algspace$ reads explicitly
\begin{equation}
\label{eq:relation-cliff}
X Y \otimes_{\Uql} T = X \otimes_{\Uql} T S(Y).
\end{equation}
On the other hand, if we consider $\xi \in \spinbun$ then we have
\[
(X Y \otimes T) \xi = (X \otimes T S(Y)) \xi,
\]
due to the defining property of $\spinbun$. Therefore $\algspace$ provides an algebraic model for operators on $\spinbun$. We will also use the relation $\sim$ from \cref{def:equiv-levi}, naturally extended to $\algspace$.

In the following we will omit the subscript $\Uql$ from the tensor product for readability.

\begin{proposition}
\label{prop:casimir-clifford}
Consider $\Cas \otimes 1 \in \algspace$. Then we have the identity
\[
\begin{split}
\Cas \otimes 1 & \sim E_{\xi_1}^* E_{\xi_1} \otimes [2]_q^2 q^{-4} K_{2 \lambda_1} \\
& + E_{\xi_2}^* E_{\xi_2} \otimes ((q^{-5} - Q q^{-2}) K_{2 \lambda_1} + q^{-1} K_{2 \lambda_2} + \qdiff^2 q^{-4} K_{2 \lambda_1} S(E_1^* E_1)) \\
& + E_{\xi_3}^* E_{\xi_3} \otimes [2]_q^2 (q^{-4} K_{2 \lambda_2} - \qdiff q^{-5} K_{2 \lambda_1} + \qdiff^2 q^{-7} K_{2 \lambda_1} S(E_1^* E_1)) \\
& - E_{\xi_1}^* E_{\xi_2} \otimes Q [2]_q q^{-3} K_{2 \lambda_1} S(E_1) - E_{\xi_2}^* E_{\xi_1} \otimes Q [2]_q q^{-3} (K_{2 \lambda_1} S(E_1))^* \\
& - E_{\xi_2}^* E_{\xi_3} \otimes Q [2]_q q^{-5} K_{2 \lambda_1} S(E_1) - E_{\xi_3}^* E_{\xi_2} \otimes Q [2]_q q^{-5} (K_{2 \lambda_1} S(E_1))^*.
\end{split}
\]
\end{proposition}

\begin{proof}
Starting from \cref{prop:cas-to-the-right} and using the relation $\sim$ we get
\[
\begin{split}
\Cas \otimes 1 & \sim [2]_q^2 E_{\xi_1}^* E_{\xi_1} q^{-4} K_{2 \lambda_1}^{-1} \otimes 1 \\
& + E_{\xi_2}^* E_{\xi_2} (q^{-7} K_{2 \lambda_1}^{-1} + q^{-1} K_{2 \lambda_2}^{-1} - \qdiff^2 [2]_q q^{-4} K_{2 \lambda_1}^{-1} + \qdiff^2 q^{-4} E_1^* E_1 K_{2 \lambda_1}^{-1}) \otimes 1 \\
& + [2]_q^2 E_{\xi_3}^* E_{\xi_3} (q^{-4} K_{2 \lambda_2}^{-1} - \qdiff q^{-5} K_{2 \lambda_1}^{-1} + \qdiff^2 q^{-7} E_1^* E_1 K_{2 \lambda_1}^{-1}) \otimes 1 \\
& - Q [2]_q q^{-5} (q^2 E_{\xi_1}^* E_{\xi_2} E_1 + q^2 E_{\xi_2}^* E_{\xi_1} E_1^* + E_{\xi_2}^* E_{\xi_3} E_1 + E_{\xi_3}^* E_{\xi_2} E_1^*) K_{2 \lambda_1}^{-1} \otimes 1.
\end{split}
\]
Next we use the defining relation of $\algspace$, that is \eqref{eq:relation-cliff}. Then we obtain
\[
\begin{split}
\Cas \otimes 1 & \sim E_{\xi_1}^* E_{\xi_1} \otimes [2]_q^2 q^{-4} K_{2 \lambda_1} \\
& + E_{\xi_2}^* E_{\xi_2} \otimes (q^{-7} K_{2 \lambda_1} + q^{-1} K_{2 \lambda_2} - \qdiff^2 [2]_q q^{-4} K_{2 \lambda_1} + \qdiff^2 q^{-4} K_{2 \lambda_1} S(E_1^* E_1)) \\
& + E_{\xi_3}^* E_{\xi_3} \otimes [2]_q^2 (q^{-4} K_{2 \lambda_2} - \qdiff q^{-5} K_{2 \lambda_1} + \qdiff^2 q^{-7} K_{2 \lambda_1} S(E_1^* E_1)) \\
& - E_{\xi_1}^* E_{\xi_2} \otimes Q [2]_q q^{-3} K_{2 \lambda_1} S(E_1) - E_{\xi_2}^* E_{\xi_1} \otimes Q [2]_q q^{-3} K_{2 \lambda_1} S(E_1^*) \\
& - E_{\xi_2}^* E_{\xi_3} \otimes Q [2]_q q^{-5} K_{2 \lambda_1} S(E_1) - E_{\xi_3}^* E_{\xi_2} \otimes Q [2]_q q^{-5} K_{2 \lambda_1} S(E_1^*).
\end{split}
\]
Finally using the identity $K_{2 \lambda_1} S(E_1^*) = (K_{2 \lambda_1} S(E_1))^* $ we get the result.
\end{proof}

\begin{remark}
As already mentioned in the introduction, the definition of the space $\algspace$ is the main difference with respect to \cite{mat-par}. Trying to compare $D^2$ with $\Cas \otimes 1$ within $\Uqg \otimes \mathrm{End}(\Lq)$ does not lead to a positive result, as shown in the cited paper.
\end{remark}

\section{More about the exterior algebra}
\label{sec:more-exterior}

In this section we derive explicit formulae for the action of $\Lq$ on itself, as well as the action of the Levi factor $\Uql$ on $\Lq$. These formulae will be used in the next section to derive various identities for the Clifford algebra $\mathrm{End}(\Lq)$.

We begin by fixing a basis of the exterior algebra $\Lq$. We have the element $1$ in degree $0$, the elements $\{ y_i \}_{i = 1}^3$ in degree $1$ and in higher degrees we take
\[
y_{2 1} := y_2 \wedge y_1, \quad
y_{3 1} := y_3 \wedge y_1, \quad
y_{3 2} := y_3 \wedge y_2, \quad
y_{3 2 1} := y_3 \wedge y_2 \wedge y_1.
\]

\begin{lemma}
\label{lem:levi-Lq}
The action of $\Uql$ on $\Lq$ is given by
\begin{center}
{
\renewcommand{\arraystretch}{1.3}
\begin{tabular}{c | c | c | c | c | c | c | c | c}
& $1$ & $y_1$ & $y_2$ & $y_3$ & $y_{2 1}$ & $y_{3 1}$ & $y_{3 2}$ & $y_{3 2 1}$ \\
\hline
$K_1$ & $1$ & $q^{-2} y_1$ & $y_2$ & $q^2 y_3$ & $q^{-2} y_{2 1}$ & $y_{3 1}$ & $q^2 y_{3 2}$ & $y_{3 2 1}$ \\
\hline
$K_2$ & $1$ & $y_1$ & $q^{-2} y_2$ & $q^{-4} y_3$ & $q^{-2} y_{2 1}$ & $q^{-4} y_{3 1}$ & $q^{-6} y_{3 2}$ & $q^{-6} y_{3 2 1}$ \\
\hline
$E_1$ & $0$ & $-[2]_q y_2$ & $- q^2 y_3$ & $0$ & $- y_{3 1}$ & $-[2]_q q^2 y_{3 2}$ & $0$ & $0$ \\
\hline
$F_1$ & $0$ & $0$ & $- y_1$ & $-[2]_q q^{-2} y_2$ & $0$ & $-[2]_q y_{2 1}$ & $- q^{-2} y_{3 1}$ & $0$
\end{tabular}
}
\end{center}
\end{lemma}

\begin{proof}
In degree $0$ we have the trivial $\Uql$-module. In degree $1$ we have already computed the action of $\Uql$ on $\lieum$ in \cref{lem:levi-um}. For higher degrees we can compute the action by using the fact that $\Lq$ is a $\Uql$-module algebra, that is $X (y \wedge y^\prime) = (X_{(1)} y) \wedge (X_{(2)} y^\prime)$. The result then follows from simple computations that we omit.
\end{proof}

\begin{corollary}
\label{cor:iso-exterior}
We have an isomorphism $\Lambda_q^1(\lieum) \cong \Lambda_q^2(\lieum)$ of $U_q(\mathfrak{l}_{ss})$-modules given by
\[
y_1 \mapsto y_{2 1}, \quad
y_2 \mapsto \frac{1}{[2]_q} y_{3 1}, \quad
y_3 \mapsto y_{3 2}.
\]
\end{corollary}

\begin{proof}
This is an immediate verification using the formulae in \cref{lem:levi-Lq}.
\end{proof}

\begin{remark}
Notice that this is an isomorphism of $U_q(\mathfrak{l}_{ss})$-modules but not of $\Uql$-modules, since for instance the action of $K_2$ on the elements $y_1$ and $y_{2 1}$ is different.
\end{remark}

Next we want to introduce a Hermitian inner product on $\Lq$, which will allow us to define adjoints of elements in $\mathrm{End}(\Lq)$.
It can be seen from \cref{lem:levi-Lq} that the subspaces $\Lambda_q^k(\lieum)$ are simple $\Uql$-modules. Hence in each degree there is a unique (up to a constant) Hermitian inner product $(\cdot, \cdot) : \Lambda_q^k(\lieum) \otimes \Lambda_q^k(\lieum) \to \mathbb{C}$, while components of different degrees are orthogonal. We determine these in the next lemma.

\begin{lemma}
\label{lem:inner-prod}
The $\Uql$-invariant inner products on $\Lq$ are given in terms of four non-zero coefficients $\{ c_k \}_{k = 0}^3$ as follows.
In degree $0$ we have $(1, 1) = c_0$. In degree $1$ we have
\[
(y_1, y_1) = c_1, \quad (y_2, y_2) = \frac{c_1}{[2]_q}, \quad (y_3, y_3) = c_1 q^{-2}.
\]
In degree $2$ we have
\[
(y_{2 1}, y_{2 1}) = c_2, \quad (y_{3 1}, y_{3 1}) = c_2 [2]_q, \quad (y_{3 2}, y_{3 2}) = c_2 q^{-2}.
\]
Finally in degree $3$ we have $(y_{3 2 1}, y_{3 2 1}) = c_3$.
\end{lemma}

\begin{proof}
The statement in degrees $0$ and $3$ is trivial, since they are one-dimensional. Also observe that elements of different weights are orthogonal, since we must have $(K_\lambda y, y^\prime) = (y, K_\lambda y^\prime)$.
Now consider degree $1$ and write $(y_1, y_1) = c_1$. We will use the formulae for the action of $\Uql$ given in \cref{lem:levi-Lq}. Then using $E_1^* = F_1 K_1 = q^2 K_1 F_1$ we compute
\[
(y_2, y_2) = - \frac{1}{[2]_q} (E_1 y_1, y_2) = - \frac{q^2}{[2]_q} (y_1, K_1 F_1 y_2) = \frac{q^2}{[2]_q} q^{-2} (y_1, y_1) = \frac{c_1}{[2]_q}.
\]
Similarly we have
\[
(y_3, y_3) = - q^{-2} (E_1 y_2, y_3) = - (y_2, K_1 F_1 y_3) = [2]_q q^{-2} (y_2, y_2) = q^{-2} c_1.
\]
Finally consider degree $2$. Then we can use the isomorphism $\Lambda_q^1(\lieum) \cong \Lambda_q^2(\lieum)$ of $U_q(\mathfrak{l}_{ss})$-modules given in \cref{cor:iso-exterior} to obtain the result from the degree $1$ case.
\end{proof}

Next we will determine explicit formulae for the action of the elements $\gamma(y_i)$ and their adjoints on $\Lq$. These results follow from tedious but straightforward computations, hence we will skip most of the details for the sake of brevity.
Also recall that $\gamma$ stands for the right regular representation, that is $\gamma(y) y^\prime = y^\prime \wedge y$.

\begin{lemma}
\label{lem:action-gamma}
The action of the $\gamma(y_i)$ on $\Lq$ is given by
\begin{center}
{
\renewcommand{\arraystretch}{1.3}
\begin{tabular}{c | c | c | c | c | c | c | c | c}
& $1$ & $y_1$ & $y_2$ & $y_3$ & $y_{2 1}$ & $y_{3 1}$ & $y_{3 2}$ & $y_{3 2 1}$ \\
\hline
$\gamma(y_1)$ & $y_1$ & $0$ & $y_{2 1}$ & $y_{3 1}$ & $0$ & $0$ & $y_{3 2 1}$ & $0$ \\
\hline
$\gamma(y_2)$ & $y_2$ & $-q^2 y_{2 1}$ & $-\qdiff [2]_q^{-1} q y_{3 1}$ & $y_{3 2}$ & $0$ & $-q^2 y_{3 2 1}$ & $0$ & $0$ \\
\hline
$\gamma(y_3)$ & $y_3$ & $-y_{3 1}$ & $-q^2 y_{3 2}$ & $0$ & $q^2 y_{3 2 1}$ & $0$ & $0$ & $0$
\end{tabular}
}
\end{center}
\end{lemma}

\begin{proof}
Follows from straightforward computations.
\end{proof}

Finally we consider the action of the adjoints $\gamma(y_i)^*$. This action only depends on the ratio of the coefficients $\{c_k\}_{k = 0}^3$, hence we introduce the notation $\kappa_k := c_k / c_{k - 1}$ for $k = 1, 2, 3$.

\begin{lemma}
The action of the $\gamma(y_i)^*$ on $\Lq$ is given by
\begin{center}
{
\renewcommand{\arraystretch}{1.3}
\begin{tabular}{c | c | c | c | c | c | c | c | c}
& $1$ & $y_1$ & $y_2$ & $y_3$ & $y_{2 1}$ & $y_{3 1}$ & $y_{3 2}$ & $y_{3 2 1}$ \\
\hline
$\gamma(y_1)^*$ & $0$ & $\kappa_1 1$ & $0$ & $0$ & $\kappa_2 [2]_q y_2$ & $\kappa_2 [2]_q q^2 y_3$ & $0$ & $\kappa_3 q^2 y_{3 2}$ \\
\hline
$\gamma(y_2)^*$ & $0$ & $0$ & $\kappa_1 [2]_q^{-1} 1$ & $0$ & $-\kappa_2 q^2 y_1$ & $-\kappa_2 \qdiff [2]_q q y_2$ & $\kappa_2 y_3$ & $-\kappa_3 [2]_q^{-1} q^2 y_{3 1}$ \\
\hline
$\gamma(y_3)^*$ & $0$ & $0$ & $0$ & $\kappa_1 q^{-2} 1$ & $0$ & $- \kappa_2 [2]_q y_1$ & $-\kappa_2 [2]_q y_2$ & $\kappa_3 q^2 y_{2 1}$
\end{tabular}
}
\end{center}
\end{lemma}

\begin{proof}
Let $y \in \Lambda_q^{k - 1}(\lieum)$ and $y^\prime \in \Lambda_q^k(\lieum)$. Then we must have $(\gamma(y_i) y, y^\prime)_k = (y, \gamma(y_i)^* y^\prime)_{k - 1}$. Then the formulae follow from explicit computations using \cref{lem:action-gamma}.
\end{proof}

\section{The Parthasarathy formula}
\label{sec:parthasarathy}

In this section we will compare the Dolbeault--Dirac element $D^2$ with the Casimir $\Cas \otimes 1$. What we are after is a relation of the form $D^2 \sim \Cas \otimes 1$, up to a constant. In the first part we show that essentially fixes the inner product on $\Lq$. In the second part we show that we get such a relation, hence a quantum version of the Parthasarathy formula.

\subsection{Vanishing of terms}

In the next lemma, we observe that having a relation of the form $D^2 \sim \Cas \otimes 1$ requires the vanishing of certain terms in the expression for $D^2$.

\begin{lemma}
A necessary condition to have $D^2 \sim \Cas \otimes 1$ (up to a constant) is that
\[
\gamma(y_1)^* \gamma(y_3) + \gamma(y_3) \gamma(y_1)^* = 0.
\]
This condition is satisfied if and only if
\[
\kappa_2 = \kappa_1 [2]_q^{-1} q^{-2}, \quad \kappa_3 = \kappa_1 q^{-4}.
\]
\end{lemma}

\begin{proof}
Consider the expression for $\Cas \otimes 1$ given in \cref{prop:casimir-clifford}: observe that the term $E_{\xi_1}^* E_{\xi_3}$ does not appear and that the elements $\{ E_{\xi_i}^* E_{\xi_j} \}_{i, j}$ are linearly independent. On the other hand consider the expression for $D^2$ given in \cref{prop:D-squared}: we see that the term with $E_{\xi_1}^* E_{\xi_3}$ in the first leg vanishes if and only if $\Gamma_{1 3} = 0$, that is $\gamma(y_1)^* \gamma(y_3) + \gamma(y_3) \gamma(y_1)^* = 0$.

Therefore we have to check when this identity is satisfied. This is the case when acting on elements of degree $0$ and $3$ in $\Lq$. On the other hand in degree $1$ we have
\[
\gamma(y_1)^* \gamma(y_3) y_1 + \gamma(y_3) \gamma(y_1)^* y_1 = -\kappa_2 [2]_q q^2 y_3 + \kappa_1 y_3.
\]
We see that we must have $\kappa_2 = \kappa_1 [2]_q^{-1} q^{-2}$. Similarly in degree $2$ we have
\[
\gamma(y_1)^* \gamma(y_3) y_{2 1} + \gamma(y_3) \gamma(y_1)^* y_{2 1} = \kappa_3 q^4 y_{3 2} - \kappa_2 [2]_q q^2 y_{3 2}.
\]
Hence we get $\kappa_3 = \kappa_2 [2]_q q^{-2} = \kappa_1 q^{-4}$. Similar computations show that the condition is also satisfied when acting on the remaining basis elements.
\end{proof}

This condition fixes two of the three free parameters $\{\kappa_k\}_{k = 1}^3$, where we recall that $\kappa_k = c_k / c_{k - 1}$.  The remaining parameter $\kappa_1$ is essentially irrelevant in what follows, as it can be absorbed as a prefactor either in the Dolbeault-Dirac element $D$ or in the Casimir $\Cas$.

To continue with the comparison of $D^2$ and $\Cas \otimes 1$, we need to simplify all the elements $\Gamma_{i j}$ defined in \cref{prop:D-squared}.
We start by considering the case $i < j$.

\begin{lemma}
\label{lem:clifford-off}
We have the relations
\[
\Gamma_{1 2} = - \kappa_1 \qdiff \frac{q}{[2]_q} K_{2 \lambda_1} S(E_1), \quad
\Gamma_{1 3} = 0, \quad
\Gamma_{2 3} = - \kappa_1 Q \frac{q^{-1}}{[2]_q} K_{2 \lambda_1} S(E_1).
\]
\end{lemma}

\begin{proof}
Follows from tedious but straightforward computations.
\end{proof}

Finally we consider the case $i = j$, which is a bit more involved.

\begin{lemma}
\label{lem:clifford-diag}
We have the relations
\[
\begin{split}
\Gamma_{1 1} & = \kappa_1 K_{2 \lambda_1}, \\
\Gamma_{2 2} & = \frac{\kappa_1}{[2]_q^2} \left( q^{-3} K_{2 \lambda_1} + q^3 K_{2 \lambda_2} - \qdiff^2 [2]_q K_{2 \lambda_1} + \qdiff^2 K_{2 \lambda_1} S(E_1^* E_1) \right), \\
\Gamma_{3 3} & = \kappa_1 \left( K_{2 \lambda_2} - \qdiff q^{-1} K_{2 \lambda_1} + \qdiff^2 q^{-3} K_{2 \lambda_1} S(E_1^* E_1) \right).
\end{split}
\]
\end{lemma}

\begin{proof}
Also follows from tedious computations, where we use $S(E_1^* E_1) = q^2 K_1^{-1} E_1 F_1$.
\end{proof}

\subsection{Comparison}

At this point we have all the necessary ingredients to compare the square of the Dolbeault--Dirac element $D^2$ with the Casimir $\Cas \otimes 1$.

\begin{theorem}
\label{thm:D-parthasarathy}
The element $D^2$ coincides with $\Cas \otimes 1$ as operators on the space $\spinbun$, up to an overall constant and terms in the quantized Levi factor $\Uql$.
\end{theorem}

\begin{proof}
We plug the relations from \cref{lem:clifford-off} and \cref{lem:clifford-diag} into the expression for $D^2$ given in \cref{prop:D-squared}.
Using these we obtain
\[
\begin{split}
D^2 & \sim E_{\xi_1}^* E_{\xi_1} \otimes \kappa_1 K_{2 \lambda_1} \\
& + E_{\xi_2}^* E_{\xi_2} \otimes \frac{\kappa_1}{[2]_q^2} \left( q^{-3} K_{2 \lambda_1} + q^3 K_{2 \lambda_2} - \qdiff^2 [2]_q K_{2 \lambda_1} + \qdiff^2 K_{2 \lambda_1} S(E_1^* E_1) \right) \\
& + E_{\xi_3}^* E_{\xi_3} \otimes \kappa_1 \left( K_{2 \lambda_2} - \qdiff q^{-1} K_{2 \lambda_1} + \qdiff^2 q^{-3} K_{2 \lambda_1} S(E_1^* E_1) \right) \\
& - E_{\xi_1}^* E_{\xi_2} \otimes \kappa_1 \qdiff [2]_q^{-1} q K_{2 \lambda_1} S(E_1) - E_{\xi_2}^* E_{\xi_3} \otimes \kappa_1 Q [2]_q^{-1} q^{-1} K_{2 \lambda_1} S(E_1) \\
& - E_{\xi_2}^* E_{\xi_1} \otimes \kappa_1 \qdiff [2]_q^{-1} q (K_{2 \lambda_1} S(E_1))^* - E_{\xi_3}^* E_{\xi_2} \otimes \kappa_1 Q [2]_q^{-1} q^{-1} (K_{2 \lambda_1} S(E_1))^*.
\end{split}
\]
Then comparing with $\Cas \otimes 1$ from \cref{prop:casimir-clifford} we see that $D^2 \sim \kappa_1 \frac{q^4}{[2]_q^2} \Cas \otimes 1$.
\end{proof}

Hence, either fixing the constant $\kappa_1$ or redefining $\Cas$, we can write
\[
D^2 \sim \Cas \otimes 1,
\]
which is the quantum analogue of the Parthasarathy formula we were looking for.

\section{Application to spectral triples}
\label{sec:spectral-triple}

In this section we will show that we get a spectral triple for the Lagrangian Grassmannian $\LagGr$.
As mentioned in the introduction, most of the steps necessary to build spectral triples on quantum flag manifolds are already given in the paper \cite{dirac-flag}. For this reason we will be very brief and concentrate on the property of compact resolvent for $D$.

We consider the Hilbert space completion of $\spinbun$, denoted by the same symbol, where the inner product is the Haar state on $\Cqg$ and the inner product $(\cdot, \cdot)$ we introduced for $\Lq$. The algebra $\LagGr$ acts by left multiplication.
The Dolbeault--Dirac operator $D = \dolb + \dolb^*$ is clearly symmetric and is easily seen to be essentially self-adjoint.

The condition of bounded commutators was proven in general in \cite{dirac-flag}. It can be proven in the same fashion also within our setting. Hence the only thing left to prove, in order to get a spectral triple for $\LagGr$, is that $D$ has compact resolvent.

\begin{theorem}
\label{thm:spectral-triple}
The Dolbeault--Dirac operator $D$ has compact resolvent. Hence we get a spectral triple for the quantum Lagrangian Grassmannian $\LagGr$.
\end{theorem}

\begin{proof}
As $D$ is self-adjoint, it is equivalent to prove that $D^* D = D^2$ has compact resolvent. From \cref{thm:D-parthasarathy} we have $D^2 \sim \Cas \otimes 1$ (possibly up to a constant). The elements of $\Uql \otimes \mathrm{End}(\Lq)$ act as bounded operators on $\spinbun$, that is $D^2$ differs from $\Cas \otimes 1$ by a bounded perturbation. Hence it suffices to prove that the latter operator has compact resolvent.

We need to recall some properties of the Hilbert space $\spinbun$.
From the Peter-Weyl decomposition $\mathbb{C}_q[G] = \bigoplus_{\Lambda \in P^+} \mathbb{C}_q[G]_\Lambda$, it follows that we have a decomposition $\spinbun = \bigoplus_{\Lambda \in P^+} \spinbun_\Lambda$.
Next, it follows from the properties of the Haar state that $(\spinbun_\Lambda, \spinbun_{\Lambda^\prime}) = 0$ for $\Lambda \neq \Lambda^\prime$. Hence we obtain an orthonormal basis of $\spinbun$ by choosing an orthonormal basis for each component $\spinbun_\Lambda$. Finally we have $(\Cas \otimes 1) a = c_\Lambda a$ for any $a \in \spinbun_\Lambda$.

Recall the following criterion to show compactness: $T \in B(H)$ is compact if we can find an orthonormal basis $\{ e_n \}_{n \in \mathbb{N}}$ such that $T e_n = \lambda_n e_n$ and $\lim_{n \to \infty} \lambda_n = 0$.
To obtain an indexing by $\mathbb{N}$ we order the weights $\Lambda = n_1 \omega_1 + n_2 \omega_2$ lexicographically: we have $(m_1, m_2) < (n_1, n_2)$ if either $m_1 < n_1$ or $m_1 = n_1$ and $m_2 < n_2$. Then, after picking any order for the orthonormal basis of each subspace $\spinbun_\Lambda$, we obtain an orthonormal basis $\{ e_n \}_{n \in \mathbb{N}}$ of $\spinbun$. Since $(\Cas \otimes 1) e_n = c_\Lambda e_n$ for $e_n \in \spinbun_\Lambda$, to prove that $\Cas \otimes 1$ has compact resolvent it suffices to show that $c_\Lambda \to \infty$ for $\Lambda \to \infty$ in the lexicographic order.

Recall that the value of $c_\Lambda$ is given in \cref{cor:value-casimir}. Since $c_\Lambda$ is given as a sum of positive terms, it suffices to show that one of them goes to infinity.
Considering for instance the term with $(\lambda_1, \Lambda) = n_1 + n_2$, it is immediate to see that $c_\Lambda \to \infty$ for $\Lambda \to \infty$.
\end{proof}

\end{document}